\documentclass{amsart}
\usepackage{amsmath,amsfonts,amssymb,amsthm,mathtools,tensor}
\usepackage[a4paper, margin=1in]{geometry} 

\usepackage{hyperref}
\usepackage{url}
\usepackage{xcolor}
\usepackage{algorithm}
\usepackage{algorithmic}
\usepackage{mathtools}
\mathtoolsset{showonlyrefs}

\newcommand{\inner}[2]{\left\langle #1, #2 \right\rangle}

\DeclareMathOperator*{\argmin}{arg\,min}

\title[Mirror descent actor-critic for entropy-regularised MDPs]{Mirror descent actor-critic methods for entropy regularised MDPs in general spaces: stability and convergence}

\addtocounter{footnote}{0}
\author{Denis Zorba}
\address{School of Mathematics, University of Edinburgh, UK}
\email{ezorba@ed.ac.uk}

\author{David \v{S}i\v{s}ka}
\address{School of Mathematics, University of Edinburgh, UK}
\email{d.siska@ed.ac.uk}

\author{Lukasz Szpruch}
\address{School of Mathematics, University of Edinburgh, UK, 
The Alan Turing Institute, UK, and Simtopia, UK}
\email{l.szpruch@ed.ac.uk}

\theoremstyle{plain}
\newtheorem{theorem}{Theorem}[section]

\newtheorem{lemma}[theorem]{Lemma}
\newtheorem{corollary}[theorem]{Corollary}
\theoremstyle{definition}
\newtheorem{definition}[theorem]{Definition}
\newtheorem{assumption}[theorem]{Assumption}
\theoremstyle{remark}
\newtheorem{remark}[theorem]{Remark}

\begin{document}

\maketitle


\begin{abstract}
We provide theoretical guarantees for convergence of discrete-time policy mirror descent with inexact advantage functions updated using temporal difference (TD) learning 
for entropy regularised MDPs in Polish state and action spaces.
We rigorously derive sufficient conditions under which the single-loop actor-critic scheme is stable and convergent.
To weaken these conditions, we introduce a variant that performs multiple TD steps per policy update
and derive an explicit lower bound on the number of TD steps required to ensure stability.
Finally, we establish sub-linear convergence when the number of TD steps grows logarithmically with the number of policy updates, and linear convergence when it grows linearly under 
a concentrability assumption. 
\end{abstract}

\section{Introduction}
In reinforcement learning (RL), an agent seeks to find an optimal policy that minimises their expected cumulative cost by interacting with their environment. 
Such a framework has had numerous remarkable applications \cite{silver}, \cite{isaacgym}, \cite{nature}. 
Theoretical understanding of key RL algorithms and their convergence properties has been growing over the last decades~\cite{jaakkola1993convergence, tsitsiklis1994asynchronous, mei2020global, agarwal2021theory}.

In practice actor critic algorithms are some of the most widely used. 
The actor updates the policy using a policy gradient method (which requires the advantage function) and the critic, which provides an estimate of the advantage function, is updated based on costs and new states observed by interacting directly with the environment using e.g. temporal difference (TD)-based loss.

Classical policy gradient methods require careful tuning of step-size to achieve stability which led to the development of trust-region approaches such as TRPO~\cite{schulman2015trust} and its practical variant PPO~\cite{schulman2017proximal}.
These improve stability by controlling the change between successive policies via KL-divergence constraint or policy ratio clipping respectively.
Policy mirror descent~\cite{lan,lan2023policy,david_fisher} provides a closely related first-order viewpoint, replacing the hard trust-region constraint with a KL-based penalty and yielding algorithms for which it is possible to obtain theoretical guarantees (under the assumption of access to the exact advantage function).
Adding entropic regularisation to the objective then leads to a number of theoretical and practical advantages~\cite{entropy_pmlr, cayci,david_leahy}. 
Firstly, entropy regularised MDPs are guaranteed to have a unique optimal policy and can accelerate the convergence of policy gradient methods \cite{mei2020global,lan,lan2023policy,david_fisher}. 
Secondly, the entropic regularisation ensures persistent exploration~\cite{rawlik2012stochastic, haarnoja2017reinforcement} and provides some (albeit minimal) robustness to changes in costs and environment~\cite{ziebart2010modeling, eysenbach2021maximum}. 

However, while the theory of discrete time policy mirror descent is well understood for MDPs without entropy regularisation \cite{tomar2022mirror} and for tabular MDPs with entropy regularisation \cite{lan2023policy}, its stability and convergence remains open for entropy regularised MDPs in general state and action spaces.
This paper fills this gap in the literature by rigorously proving the stability and convergence of policy mirror descent for entropy regularised MDPs in Polish state and action spaces, where the critic is updated using one step or multiple steps of temporal difference learning. 

\subsection{Related works}
While actor-critic methods employing mirror descent for unregularised MDPs are well understood, both with and without function approximation \cite{tomar2022mirror}, \cite{qiu2021finite_approxerror1}, the literature surrounding its extension to entropy regularised MDPs still remains sparse.
In this setting, \cite{lan2023policy}, \cite{cayci} demonstrates convergence through carefully constructed decaying step sizes  tabular case.
More recently, \cite{zorba2025convergenceactorcriticentropyregularised} extend these results to the Polish state and action space case.
However, this paper uses the idealised continuous-time dynamics of the policy mirror descent and temporal difference updates.
While the continuous-time limit is mathematically convenient, its connection to any given algorithm is unclear as there are  many discrete schemes that converge to the same continuous-time limit as step size goes to zero.
However, the convergence, or even stability, as the number of discrete steps goes to infinity is not guaranteed.
Moreover, the convergence of the dynamics in~\cite{zorba2025convergenceactorcriticentropyregularised} depends on an exponentially increasing timescale separation.
While this is well defined in the continuous time limit, it is unclear what it implies in relation to the discrete time algorithm.

\subsection{Contributions}\label{sec:contributions}
Our main contributions are as follows.
\begin{itemize}
\item For entropy regularised MDPs in general spaces, we consider a classical actor critic algorithm where the policy is updated using mirror descent and the critic is updated using temporal difference learning. As pointed out in \cite{zorba2025convergenceactorcriticentropyregularised}, ensuring that the relative entropy does not blow up along the gradient flow is difficult in general action spaces.
In the finite action space setting, for any measure $\mu \in \mathcal{P}(A)$ such that $\mu(a_i) > 0$ and for all $s \in S$ it holds that $\operatorname{KL}(\pi(\cdot|s) |\mu) \leq \log |A|$. 
In general action spaces the $\operatorname{KL}$ divergence has no  upper bound (can be $+\infty$) even if $\mu$ has full support.
Thus we derive sufficient conditions of the discrete actor critic updates, under which we demonstrate uniform boundedness of the KL divergence along the actor-critic stepping scheme.
\item \cite{lan2023policy, david_fisher} demonstrate that classical policy mirror descent with exact advantage functions for entropy regularised MDPs in general spaces exhibits sublinear convergence and even linear convergence under a concentrability assumption.
We extend this result and demonstrate that under $Q^{\pi}_{\tau}$-realisability, one can still obtain both rates while using TD learning to update an approximation of the advantage function. 
To the best of our knowledge, these stability and convergence results are the first to be established in the literature.
\end{itemize}

\subsection{Entropy regularised Markov Decision Processes}

See Appendix \ref{sec:technical} for a summary of the notation used in this paper. 
Consider an infinite horizon Markov Decision Process $(S,A,P,c,\gamma)$, where the state space $S$ and action space $A$ are Polish, $P\in \mathcal{P}(S | S \times A)$ is the state transition probability kernel, $c$ is a bounded cost function and $\gamma \in (0,1)$ is a discount factor. Let $\mu \in \mathcal{P}(A)$ denote a reference probability measure and $\tau > 0$ denote a regularisation parameter. To ease notation, for each $\pi \in \mathcal{P}(A|S)$, $s \in S$ and $a \in A$ we define
\begin{equation}
P_{\pi}(ds'|s) := \int_{A} P(ds'|s,a)\pi(da|s),
\end{equation}
\begin{equation}
P^{\pi}(ds',da'|s,a) := P(ds'|s,a)\pi(da'|s').
\end{equation}
For each stochastic policy $\pi \in \mathcal{P}(A|S)$ and $s \in S$, we define the regularised value function by
\begin{equation}
\begin{aligned}\label{eq:value_function}
V^{\pi}_{\tau}(s)
&= \mathbb{E}_{s}^{\pi}\left[\sum_{n=0}^\infty\gamma^n
\Big(c(s_n,a_n) + \tau \operatorname{KL}(\pi(\cdot|s_n)|\mu)\Big)\right]
.
\end{aligned}
\end{equation}
Here $\operatorname{KL}(\pi(\cdot|s)|\mu)$ is the Kullback-Leibler (KL) divergence of $\pi(\cdot|s)$ with respect to $\mu$, defined as
\begin{equation}
\operatorname{KL}(\pi(\cdot|s)|\mu)
:= \int_{A} \ln \frac{d\pi}{d\mu}(a|s) \pi(da|s),
\end{equation}
if $\pi(\cdot|s)$ is absolutely continuous with respect to $\mu$, and infinity otherwise. As a result, for any $s \in S$ we have that $V^{\pi}_{\tau}(s) \in \mathbb{R}\cup \{+\infty\}$. For a given initial distribution $\rho \in \mathcal{P}(S)$, the optimal value function is defined as
\begin{equation}
\label{eq:optim}
\begin{aligned}
V^{\pi}_{\tau}(\rho) := \int_{S} V^{\pi}_{\tau}(s)\rho(ds), \quad V^{*}_{\tau}(\rho) &= \inf_{\pi \in \mathcal{P}(A|S)}V^{\pi}_{\tau}(\rho).
\end{aligned}
\end{equation}
We refer to $\pi^* \in \mathcal{P}(A|S)$ as the optimal policy if $V^*_{\tau}(\rho) = V^{\pi^*}_{\tau}(\rho)$.

The Bellman Principle for entropy regularised MDPs, see Theorem~\ref{thm:dynamics_programming}, ensures that without loss of generality, it is sufficient to consider policies from the class given by Definition~\ref{def:admissible_policies} below.

\begin{definition}[Admissible Policies]
\label{def:admissible_policies}
Let $\Pi_{\mu}$ denote the class of policies for which there exists $f \in B_{b}(S \times A)$ with
\begin{equation}
\pi(da|s) = \frac{\exp(f(s,a))}{\int_{A} \exp(f(s,a))\mu(da)}\mu(da).
\end{equation}
\end{definition}

For each $\pi \in \Pi_\mu$ the value function $V^{\pi}_{\tau}$ is the unique bounded solution of the on-policy Bellman equation
\begin{equation}
\begin{aligned}
V^{\pi}_{\tau}(s)
&= \int_{A}\left(Q_\tau^\pi(s,a)+\tau \ln \frac{d \pi}{d \mu}(a,s)\right)\pi(da|s).
\end{aligned}
\end{equation}
See e.g. Lemma~B.2 of \cite{david_fisher}. Moreover, for each $\pi \in \Pi_\mu$, the state-action value function $Q^{\pi}_{\tau}\in B_b(S\times A)$ is defined as
\begin{equation}
\label{eq:Q_func}
Q^{\pi}_{\tau}(s,a)=c(s,a)+\gamma\int_S V_{\tau}^{\pi}(s')P(ds'|s,a).
\end{equation} and the soft advantage function $A^{\pi}_{\tau} \in B_{b}(S\times A)$ as 
\begin{equation}
A^{\pi}_{\tau}(s,a) = Q^{\pi}_{\tau}(s,a) + \tau \log \frac{d\pi}{d\mu}(s,a) - V^{\pi}_{\tau}(s).
\end{equation}
Finally we define the Bellman operator $\mathrm{T}^{\pi}_{\tau} : B_{b}(S\times A) \to B_{b}(S\times A)$ as
\begin{equation}
\label{eq:bellman_operator_Def}
\begin{aligned}
\mathrm{T}^{\pi}_{\tau} f(s,a)
&= c(s,a) + \gamma\int_{S\times A} f(s',a') P^{\pi}(ds',da'|s,a)+ \tau\gamma \int_{S} \operatorname{KL}(\pi(\cdot|s')|\mu) P(ds'|s,a).
\end{aligned}
\end{equation}
A direct calculation shows that for any $\pi \in \mathcal{P}(A|S)$, $T^{\pi}_{\tau}$ is a $\gamma$-contraction on $B_{b}(S\times A)$ with respect to $| \cdot |_{B_b(S\times A)}$ and hence by Banach's fixed point theorem $Q^{\pi}_{\tau} : S \times A \to \mathbb{R}$ is the unique fixed point of $\mathrm{T}^{\pi}_{\tau}$.

\section{Mirror Descent and Temporal Difference}

In this section we introduce an actor critic scheme to find the optimal policy $\pi^* \in \Pi_{\mu}$ of the entropy regularised MDP. Firstly, let's consider the following policy mirror descent updates. For some $\lambda > 0$, define $G : \mathcal{P}(A|S) \times \mathcal{P}(A|S) \to \mathbb{R}$ as
\begin{align}
G(\pi,\pi')= \int_{S}\bigg(\int_{A} A^{\pi'}_{\tau}(s,a)\pi(da|s) + \frac{1}{\lambda}\operatorname{KL}(\pi|\pi')(s)\bigg)d_{\rho}^{\pi'}(ds),
\end{align}
and consider the updates
\begin{align}\label{eq:exact_mirror_descent}
\pi^{n+1}
&= \argmin_{\pi \in \mathcal{P}(A|S)} G(\pi,\pi^{n}).
\end{align}

In the setting where the advantage function can be calculated exactly for all $s \in S$ and $a \in A$, \cite{lan2023policy,david_fisher} demonstrates convergence of the policy mirror descent for entropy regularised MDPs in the tabular setting.
However having direct access to the true advantage function for all $s \in S$ and $a \in A$ is unfeasible in general spaces. To that end, we consider the following finite dimensional parametrisation of the advantage function. Given some feature mapping $\phi : S \times A \to \mathbb{R}^N$, we firstly parametrise the state-action value function as $Q(s,a;\theta) := \inner{\theta}{\phi(s,a)}$ and define the approximate soft advantage function as in Definition \ref{def:approximate_advantage_def}.

\begin{definition}\label{def:approximate_advantage_def}
Let $Q(s,a;\theta) := \langle \theta, \phi(s,a)\rangle$, for some $\phi : S \times A \to \mathbb{R}^N$ and any $\tau > 0$. For all $s \in S$ and $a \in A$, the approximate advantage function is defined as 
\begin{equation}
\label{eq:approximate_advantage}
\begin{aligned}
A(s,a;\theta,\pi)
= Q(s,a;\theta) + \tau \ln \frac{d\pi}{d\mu}(s,a) - \int_{A}\bigg(Q(s,a;\theta) + \tau \ln \frac{d\pi}{d\mu}(s,a) \bigg) \pi(da|s).
\end{aligned}
\end{equation}
\end{definition}

\begin{definition}
\label{def:MSBE}
For some $\beta \in \mathcal{P}(S \times A)$, let $d_{\beta}^{\pi} \in \mathcal{P}(S\times A)$ be the state-action occupancy measure defined in Appendix \ref{sec:technical}. The Mean Squared Bellman Error (MSBE) is defined as
\begin{equation}
\begin{aligned}
&\mathrm{MSBE}(\theta,\pi):= \frac{1}{2}\int_{S \times A}
(Q(s,a;\theta) - \mathrm{T}^{\pi}_{\tau}Q(s,a;\theta))^2
d_{\beta}^{\pi}(da,ds).
\end{aligned}
\end{equation}
Moreover, the semi-gradient $g: \mathbb{R}^N \times \mathcal{P}(A|S) \to \mathbb{R}^N$ of the MSBE with respect to $\theta$ is given by
\begin{equation}
\label{eq:semi_gradient_def}
\begin{aligned}
g(\theta,\pi):= \int_{S \times A}(Q(s,a;\theta) - \mathrm{T}^{\pi}_{\tau}Q(s,a;\theta))
\phi(s,a)\, d_{\beta}^{\pi}(da,ds).
\end{aligned}
\end{equation}
\end{definition}

\begin{remark}
Given that $\beta \in \mathcal{P}(S \times A)$ has full support, by \eqref{eq:bellman_operator_Def} it holds that $\mathrm{MSBE}(\theta,\pi) = 0$ if and only if $Q(s,a;\theta) = Q^{\pi}_{\tau}(s,a)$ for all $s \in S$ and $a \in A$.
\end{remark}
\begin{assumption}\label{as:e_value}
Let $\beta \in \mathcal{P}(S \times A)$ be fixed. Then
\begin{equation}
\lambda_{\beta} := \lambda_{\min}\left( \int_{S \times A} \phi(s,a)\phi(s,a)^{\top} \, \beta(ds\,da) \right) > 0.
\end{equation}
\end{assumption}
Note that unlike the analogous assumptions in the literature \cite{stochastic_approx_application_AC}, Assumption \ref{as:e_value} is independent of the policy and only depends on some $\beta \in \mathcal{P}(S \times A)$ and $\phi : S \times A \to \mathbb{R}^{N}$ which are both design choices.

\begin{assumption}\label{as:bounded_phi}
For all $(s,a)\in S \times A$ it holds that $|\phi(s,a)|_2 \leq 1$.
\end{assumption}
Assumption \ref{as:bounded_phi} is for convention and is without loss of generality in the finite-dimensional case.
\begin{assumption}[$Q^{\pi}_{\tau}$-realisability]\label{as:linearmdp}
For all $\pi \in \Pi_{\mu}$ there exists $\theta_{\pi} \in \mathbb{R}^N$ such that $Q^{\pi}(s,a) = \inner{\theta_{\pi}}{\phi(s,a)}$ for all $(s,a)\in S \times A$.
\end{assumption}
Let us comment on using linear function approximation.
While non-linear function approximation (using deep neural networks) are widely used in practice, there are no proofs of convergence even for supervised learning with deep neural networks.
This analytical intractability of non-linear function approximations leads us to adopt linear function approximation.

The $Q^\pi$ realisability (Assumption \ref{as:linearmdp}) is common in the literature \cite{sean, meyn2024projected, csaba_planning, cayci, cayci_neural, stochastic_approx_application_AC}.
One classical family of MDPs which satisfy Assumption \ref{as:linearmdp} are Linear MDPs, see e.g \cite{li2021sampleefficient, yang2019sample, zanette2020frequentist}.
Moreover, Assumption \ref{as:linearmdp} holds in the limit $N \to \infty$ when $\phi_i$ are the basis functions of $L^2(\rho \otimes \mu)$ for some $\rho \otimes \mu \in \mathcal{P}(S \times A)$ \cite{brezis}.
With this perspective in mind, \cite{ma2024skill_quadraptor, ren2023stochastic_truncation} demonstrate empirical success of linear MDPs when the feature mapping $\phi :S \times A \to \mathbb{R}^{N}$ is careful truncation of the $L^2(\rho \otimes \mu)$ basis functions.

\section{Single loop actor-critic}
\label{sec:single_loop}

At each iteration of policy mirror descent \eqref{eq:exact_mirror_descent}, one would ideally replace the true advantage function with its parametrisation $A(s,a;\theta^* ,\pi^n)$ such that $\theta^* = \min_{\theta \in \mathbb{R}^{N}} \mathrm{MSBE}(\theta,\pi^n)$. Of course computationally this is very demanding, hence it is common in practice to perform a single step of temporal difference learning on the mean-squared Bellman error and use the updated parameters in the policy mirror descent update.

That is, let
\begin{align}\label{eq:inexact_mirror_descent}
\tilde{G}(\pi,\pi',\theta) 
& = \int_{S}\bigg(\int_{A} A(s,a;\theta,\pi')\pi(da|s) + \frac{1}{\lambda}\operatorname{KL}(\pi|\pi')(s)\bigg)d_{\rho}^{\pi'}(ds),
\end{align}

then for some critic step size $h>0$ and actor step size $\lambda > 0$, initial policy and parameters $\pi^0 = \pi_0  \in \Pi_{\mu}$, $\theta^0 = \theta_0 \in \mathbb{R}^{N}$, consider Algorithm \ref{algo:single_loop_ac}.

\begin{algorithm}[H]
\caption{}
\label{algo:single_loop_ac}
\begin{algorithmic}[1]
\small
\REQUIRE Critic step size $h>0$, actor step size $\lambda>0$, initial parameters $\theta^{0}\in\mathbb{R}^{N}$, initial policy $\pi^{0}\in\Pi_\mu$
\FOR{$n = 0,1,2,\ldots$}
\STATE $\theta^{n+1} \gets \theta^{n} - h\, g(\theta^{n}, \pi^{n})$
\STATE {$ \displaystyle \pi^{n+1}\gets \argmin_{\pi \in \mathcal{P}(A|S)}\tilde{G}(\pi,\pi^n,\theta^{n+1}) $}
\ENDFOR
\end{algorithmic}
\end{algorithm}

As pointed out in Section \ref{sec:contributions}, it is still unclear when Algorithm \ref{algo:single_loop_ac} remains stable for entropy regularised MDPs in general spaces. We address this in the following section.

\subsection{Stability}

To ease notation in the main results, for each $s \in S$ and $a \in A$ we let
\begin{equation}
l_{n}(s,a) = \log \frac{d\pi^n}{d\mu}(s,a)
- \int_{A} \log \frac{d\pi^n}{d\mu}(s,a')\mu(da'),
\end{equation}
\begin{equation}
\mathrm{K}_n:= \sup_{s \in S} \operatorname{KL}(\pi^n(\cdot|s)|\mu), \quad \Gamma:= (1-\gamma)(1-\sqrt{\gamma})\lambda_{\beta}.
\end{equation}

To establish conditions for stability of Algorithm \ref{algo:single_loop_ac}, we firstly prove a coupled recursion for the critic parameters.

\begin{lemma}\label{lemma:theta_recursion_1}
Let Assumption \ref{as:e_value} and \ref{as:bounded_phi} hold. Let $0<h \leq \frac{\Gamma}{6(1+\gamma)^2}$ and for some $\theta^0 = \theta_0 \in \mathbb{R}^{N}$ and $\pi^0 = \pi_0 \in \Pi_{\mu}$, let $\{\theta^{n},\pi^n\}_{n\in \mathbb{N}}$ be the iterates for Algorithm \ref{algo:single_loop_ac}. Then for all $n \in \mathbb{N}$ it holds that
\begin{equation}
\begin{aligned}
|\theta^{n+1}|_2^2
&\leq |\theta_0|_2^2
+ \frac{\tau^2\gamma^2\left(3h + \frac{2}{\Gamma}\right)}{\Gamma - 3h(1+\gamma)^2}
\sup_{0 \leq r \leq n}\mathrm{K}_{r}^2  + \frac{|c|_{B_b(S\times A)}^2\left(3h + \frac{2}{\Gamma}\right)}{\Gamma - 3h(1+\gamma)^2}.
\end{aligned}
\end{equation}
\end{lemma}
See Appendix \ref{sec:proof_theta_recursion} for a proof.

Lemma \ref{lemma:theta_recursion_1} shows that the behaviour of the norm of the critic parameters $|\theta^{n+1}|_2$ depends linearly on $\mathrm{K}_{n}$, which may explode in general action spaces if the critic is not solved to sufficient accuracy.
On this note, a direct corollary of Lemma \ref{lemma:theta_recursion_1} is that we automatically arrive at stability for action spaces with finite cardinality when the reference measure in \eqref{eq:value_function} has full support.
\begin{corollary}\label{cor:finite_action_space}
Let Assumption \ref{as:e_value} and \ref{as:bounded_phi} hold and let $0<h \leq \frac{\Gamma}{6(1+\gamma)^2}$. Suppose that $|A| < \infty$ and let $\mu \in \mathcal{P}(A)$ be such that $\min_{a \in A} \mu(a) > 0$. Then for all $n \in \mathbb{N}$ there exists $R_1 > 0$ such that  
\begin{equation}
|\theta^{n}|_2 \leq R_1.
\end{equation}
\end{corollary}
For the more general setting of Polish action spaces, we do not directly arrive at stability from Lemma \ref{lemma:theta_recursion_1}. To that end, Lemma \ref{lemma:log_recursion} demonstrates a basic recursion for the normalised log densities which we can connect back to the KL divergence through Lemma \ref{lemma:technical_KL_ln}.

\begin{lemma}\label{lemma:log_recursion}
Let Assumption \ref{as:bounded_phi} hold. Then for all $n \in \mathbb{N}$ it holds that
\begin{equation}
|l_{n+1}|_{B_b(S\times A)}
\leq (1-\tau\lambda) |l_{n}|_{B_b(S\times A)} + 2\lambda|\theta^{n+1}|_2.
\end{equation}
\end{lemma}
See Appendix \ref{sec:proof_log_recursion} for a proof.

Thus by connecting Lemma \ref{lemma:theta_recursion_1} and Lemma \ref{lemma:log_recursion}, we can derive sufficient conditions under which we achieve stability of Algorithm \ref{algo:single_loop_ac} for Polish action spaces.

\begin{theorem}\label{thm:KL_BDD_small_gamma}
Let Assumptions \ref{as:e_value} and \ref{as:bounded_phi} hold. Let \[0 < h \leq\frac{1}{2} \min\left\{
\frac{\Gamma}{3(1+\gamma)^2},
\frac{\Gamma^2 - 16\gamma^2}{\Gamma\bigl(24\gamma^2 + 3(1+\gamma)^2\bigr)}
\right\}\] and $0 < \tau\lambda < 1$. Moreover, for some $\theta^0 = \theta_0 \in \mathbb{R}^{N}$ and $\pi^0 = \pi_0 \in \Pi_{\mu}$ let $\{\theta^n,\pi^n\}_{n \in \mathbb{N}}$ be the iterates of Algorithm \ref{algo:single_loop_ac}. 
Then there exists $R_2\geq 0$ such that for $\frac{32\gamma^2}{\Gamma^2} < 1$ and for all $n \in \mathbb{N}$ and $s \in S$ it holds that
\begin{equation}
\mathrm{KL}(\pi^n(\cdot|s) | \mu) + |\theta^{n}|_2\leq R_2.
\end{equation}
\end{theorem}
See Appendix \ref{sec:proof_KL_bdd_small_gamma} for a proof.

\begin{remark}
Observe that the results of this section, Lemma \ref{lemma:theta_recursion_1}, Lemma \ref{lemma:log_recursion} and Theorem \ref{thm:KL_BDD_small_gamma} do not require $Q^{\pi}_{\tau}$-realisability and thus holds for general MDPs given that the features $\phi : S \times A \to \mathbb{R}^N$ satisfy Assumptions \ref{as:e_value} and \ref{as:bounded_phi}.
\end{remark}
Theorem \ref{thm:KL_BDD_small_gamma} suggests performing a single step of TD learning for each mirror descent update may be insufficient to establish stability of Algorithm \ref{algo:single_loop_ac} for all $\gamma \in (0,1)$ in Polish action spaces. 
In turn, Theorem \ref{thm:KL_BDD_small_gamma} indicates that one must have more control over the temporal difference learning in order to obtain a more accurate parametrisation of the advantage function. 
With this in mind, in Section \ref{sec:double_loop} we analyse the setting where one performs $M(n) \geq 1$ steps of temporal difference instead of a single step, where the number of inner steps depend on the policy iterate $n \in \mathbb{N}$, see Algorithm \ref{algo:double_loop_ac}.
\subsection{Convergence}
To establish the convergence of Algorithm \ref{algo:single_loop_ac}, we firstly establish the following continuity property for consecutive policies under Algorithm \ref{algo:single_loop_ac}. 
\begin{theorem}\label{thm:lipschitz_holder} Suppose that there exists $R \geq 0$ such that $ \operatorname{KL}(\pi^{n}(\cdot|s)|\mu) \leq R$ for all $s \in S$ and $n \in \mathbb{N}$. Then there exists $\alpha_1, \alpha_2 > 0$ such that for all $n \in \mathbb{N}$ and $s \in S$ it holds that
\begin{align}
&\left|Q^{\pi^{n+1}}_{\tau} - Q^{\pi^{n}}_{\tau} \right|_{B_b(S \times A)} \leq  \alpha_1 \sup_{s \in S}\operatorname{KL}(\pi^{n+1}|\pi^n)(s)^{\frac{1}{2}}+ \alpha_2\sup_{s \in S}\operatorname{KL}(\pi^{n+1}|\pi^n)(s).
\end{align}
\end{theorem}
See Appendix \ref{sec:proof_lipschitz} for a proof.
Theorem \ref{thm:single_loop_conv} then demonstrates that the convergence of Algorithm \ref{algo:single_loop_ac} is determined purely by the timescale separation $\eta := \frac{\lambda}{h}$.
\begin{theorem}\label{thm:single_loop_conv}
For some $\theta^0=\theta_0\in\mathbb{R}^N$ and $\pi^0=\pi_0\in\Pi_\mu$ let $\{\theta^n,\pi^n\}_{n\in\mathbb{N}}$ be the iterates of Algorithm \ref{algo:single_loop_ac}. Suppose that the conditions of Corollary \ref{cor:finite_action_space} or Theorem \ref{thm:KL_BDD_small_gamma} hold. Then there exists a constant $C\geq0$ such that for any $\rho\in\mathcal{P}(S)$ and all $n\in\mathbb{N}$, it holds that
\begin{equation}\label{eq:single_loop_conv_rate_simple}
\min_{0\le r\le n-1}\Big(V^{\pi^r}_{\tau}(\rho)-V^{\pi^*}_{\tau}(\rho)\Big)
\leq C\bigg(
\frac{1}{\sqrt{nh}}
+
\frac{\lambda}{h}\bigg).
\end{equation}
\end{theorem}

\section{Double loop actor-critic}
\label{sec:double_loop}
With the results of Theorem \ref{thm:KL_BDD_small_gamma} in mind, we extend Algorithm \ref{algo:single_loop_ac} to incorporate $M(n) \geq 1$ temporal difference steps for each policy $\pi^n \in \mathcal{P}(A|S)$.
Let $\{\theta^{n,k}\}_{k=0}^{M(n)}$ denote the temporal difference steps for each $n \in \mathbb{N}$ and consider Algorithm \ref{algo:double_loop_ac}. 

\begin{algorithm}[H]
\caption{}
\label{algo:double_loop_ac}
\begin{algorithmic}[1]
\small
\REQUIRE Critic step size $h>0$, actor step size $\lambda>0$, number of critic steps $M\in\mathbb{N}$, 
initial parameters $\theta^{0}\in\mathbb{R}^{N}$, initial policy $\pi^{0}\in\Pi_\mu$
\FOR{$n = 0,1,2,\ldots$}
\STATE $\theta^{n,0} \gets \theta^{n}$
\FOR{$k = 1,\ldots,M(n)$}
\STATE $\theta^{n,k} \gets \theta^{n,k-1} - h\, g(\theta^{n,k-1}, \pi^{n})$
\ENDFOR
\STATE $\theta^{n+1} \gets \theta^{n,M}$
\STATE $\displaystyle \pi^{n+1} \gets 
\argmin_{\pi \in \Pi_\mu} \tilde{G}(\pi,\pi^{n},\theta^{n+1})$
\ENDFOR
\end{algorithmic}
\end{algorithm}

In this setting, we address the following fundamental question.
\begin{center}
\textit{
For each \(n \in \mathbb{N}\) and for all $\gamma \in (0,1)$, what is the minimal number of temporal-difference steps
\(M(n) \geq 1 \) required to ensure stability of Algorithm~\ref{algo:double_loop_ac} for entropy regularised MDPs in Polish action spaces?
}
\end{center}

\subsection{Stability}
To rigorously address the stability of Algorithm \ref{algo:double_loop_ac}, we firstly present some useful technical results on the semi-gradient.

\begin{lemma}\label{lemma: bound_g_squared}
Let Assumption \ref{as:e_value}, \ref{as:bounded_phi} and \ref{as:linearmdp} hold. Then for all $\theta \in \mathbb{R}^N$ and $\pi \in \Pi_{\mu}$ it holds that
\begin{equation} 
\left|g(\theta,\pi) \right|_2^2 \leq 2(1+\gamma)\left| \theta - \theta_{\pi}\right|_2^2
\end{equation}
\end{lemma}
See Appendix \ref{sec:proof_gradient_squared_bound} for a proof. For any fixed $n \in \mathbb{N}$, Theorem \ref{thm:inner_convergence} demonstrates linear convergence of the inner temporal difference loop for a sufficiently small critic step size under Assumptions \ref{as:e_value} and \ref{as:linearmdp}.
\begin{theorem}\label{thm:inner_convergence} Let Assumption \ref{as:e_value}, \ref{as:bounded_phi} and \ref{as:linearmdp} hold. Let $0 < h < \min\left\{\frac{\Gamma}{2(1+\gamma)},\frac{1}{\Gamma}\right\}$. Then for all $n \in \mathbb{N}$ it holds that
\begin{equation}
|\theta^{n+1} - \theta_{\pi^n}|_2^2 \leq e^{-M(n) h\Gamma}|\theta^{n} - \theta_{\pi^n}|_2^2.
\end{equation}
\end{theorem}
See Appendix \ref{sec:proof_inner_conv} for a proof. 
Moreover, Lemma \ref{lemma:value_bound} then demonstrates that Algorithm \ref{algo:double_loop_ac} produces policies that improves the value function up to a critic approximation error.

\begin{lemma}\label{lemma:value_bound}
Let Assumption \ref{as:e_value}, \ref{as:bounded_phi} and \ref{as:linearmdp} hold. Let $0 < h < \min\left\{\frac{\Gamma}{2(1+\gamma)},\frac{1}{\Gamma}\right\}$. Then for all $n \in \mathbb{N}$ it holds that
\begin{equation}
V^{\pi^*}_{\tau}(s)\leq V^{\pi^{n+1}}_{\tau}(s) \leq V^{\pi^{n}}_{\tau}(s) + \frac{2e^{-\frac{M(n)h\Gamma}{2}}}{1-\gamma}|\theta^{n} - \theta_{\pi^n}|_{2}
\end{equation}

\end{lemma} 
See Appendix \ref{sec:proof_value_bound} for a proof. Theorem \ref{thm:bounded_under_log_growth} then gives a sufficient growth condition on the number of inner
temporal-difference steps required to guarantee the stability of Algorithm \ref{algo:double_loop_ac} in Polish action spaces.

\begin{theorem}\label{thm:bounded_under_log_growth}
Let Assumption \ref{as:e_value}, \ref{as:bounded_phi}, and \ref{as:linearmdp} hold. Let $0 < h < \min\left\{\frac{\Gamma}{2(1+\gamma)},\frac{1}{\Gamma}\right\}$ and $0 <  \tau\lambda < 1$. Moreover, for each mirror descent step $n \in \mathbb{N}$ and for some $ c \geq 0$ let the number of inner temporal difference steps satisfy 
\begin{equation}
M(n) \geq \frac{4}{h\Gamma}\log\left(c(n+1) \right)
\end{equation}Then there exists $R \geq 0 $ such that for all $s \in S$ and $n \in \mathbb{N}$ it holds that
\begin{equation}
|\theta^{n}|_2 + \operatorname{KL}(\pi^{n}(\cdot|s) | \mu) \leq R.
\end{equation}
\end{theorem}See Appendix \ref{sec:proof_bounded_under_log_growth} for a proof.

Theorem \ref{lemma:value_bound} shows that in order to achieve stability of Algorithm \ref{algo:double_loop_ac} for entropy regularised MDPs in general Polish action spaces, one has to increase the number of temporal difference steps logarithmically with the number of mirror descent updates.
This in turn implies that to perform $n$ policy updates one takes $\mathcal O(n \log n)$ TD learning steps. 
The convergence rate reported in Theorems~\ref{thm:linear_convergence_log_growth} and~\ref{thm:exponential_conv} has to be read in this context: it is stated in terms of the number of policy updates, not the number of TD learning steps.

\subsection{Convergence}\label{sec:convergence_double_loop}
Theorem \ref{thm:linear_convergence_log_growth} shows that if the number of inner temporal difference steps satisfies the conditions of Theorem \ref{thm:bounded_under_log_growth}, Algorithm \ref{algo:double_loop_ac} converges sub-linearly to the optimal value function for a sufficiently small critic step size.

\begin{theorem}
\label{thm:linear_convergence_log_growth}
Let Assumption \ref{as:linearmdp}, \ref{as:e_value} and \ref{as:bounded_phi} hold and let $0 < h < \frac{\Gamma}{2(1+\gamma)}$. Moreover for each $ n\in \mathbb{N}$, let the number of inner critic steps satisfy $M(n) \geq \frac{4}{h\Gamma}\log\left(c(n+1)\right)$ with $c > 0$ the constant from Theorem \ref{thm:bounded_under_log_growth}. Then there exists $a \geq 0$ such that for all $n \in \mathbb{N}$ it holds that
\begin{equation}
\min_{0 \leq r \leq n-1} V^{\pi^r}_{\tau}(\rho) - V^{\pi^*}_{\tau}(\rho) \leq \frac{a}{n}.
\end{equation}
\end{theorem}
See Appendix \ref{sec:proof_linear_conv_log_growth} for a proof. Finally, Theorem \ref{thm:exponential_conv} demonstrates that under a concentrability assumption, if we increase the number of temporal difference steps linearly with $n \in \mathbb{N}$, we obtain linear convergence of Algorithm \ref{algo:double_loop_ac}.

\begin{theorem}\label{thm:exponential_conv}
Let Assumption \ref{as:linearmdp}, \ref{as:e_value} and \ref{as:bounded_phi} hold and fix $\rho \in \mathcal{P}(S)$.
Moreover suppose that $\Big|\frac{\mathrm{d}d_{\rho}^{\pi^*}}{\mathrm{d}\rho} \Big|_{B_b(S)} \leq \xi < \infty$.
Let $0 < h < \frac{\Gamma}{2(1+\gamma)}$ and $0 \leq \tau \lambda \leq \frac{\xi - 1}{\xi}$
and for each policy update $n \in \mathbb{N}$ let the number of inner temporal difference steps satisfy $M(n)\geq \frac{4c}{h\Gamma}(n+1)$ with $c \geq 0$ the constant from Theorem \ref{thm:bounded_under_log_growth}.
Then there exists $b \geq 0$ such that for all $ n \in \mathbb{N}$ it holds that
\begin{align}
V^{\pi^n}_{\tau}(\rho) - V^{\pi^*}_{\tau}(\rho)
&\le be^{-\min\left\{ \frac{1}{\xi},\, c\right\}n}.
\end{align}
See Appendix \ref{sec:proof_exp_conv} for a proof.
\end{theorem}

\section{Conclusion and future directions}
In this work, we study stability and convergence of a fundamental actor-critic algorithm for entropy regularized MDPs on general state-action spaces under $Q^\pi$-realisability with linear function approximation.
In the algorithm the policy is updated through mirror descent and the critic parameters via TD learning. 
We provide sufficient, but restrictive, conditions guaranteeing stability and linear convergence for the single-step variant.
For the multi-step variant we remove the restrictive assumption and by appropriately controlling the number of TD updates per policy update, we establish sublinear and, under an additional concentrability assumption, linear convergence rates.
To the best of our knowledge, these results address a significant theoretical gap for actor–critic methods in entropy-regularised MDPs.
To focus on the algorithm’s behaviour in general spaces, we analyse the population setting in which all integrals are evaluated exactly.
Extending the analysis to the sample-based setting remains an important direction for future work.

\section*{Acknowledgements}
DZ was supported by the EPSRC Centre for Doctoral Training in Mathematical Modelling, Analysis and Computation (MAC-MIGS) funded by the UK Engineering and Physical Sciences Research Council (grant EP/S023291/1), Heriot-Watt University and the University of Edinburgh. We acknowledge funding from the UKRI Prosperity Partnerships grant APP43592: AI2 – Assurance and Insurance for Artificial Intelligence, which supported this work. The authors would like to thank the Isaac Newton Institute for Mathematical Sciences, Cambridge, for support and hospitality during the programme Bridging Stochastic Control And Reinforcement Learning, where work on this paper was undertaken. This work was supported by EPSRC grant EP/V521929/1.

\bibliographystyle{abbrv}
\bibliography{arxivbib}

\appendix

\section{Notation}
Let $(E, d)$ denote a Polish space (i.e., a complete separable metric space). We always equip a Polish space with its Borel sigma-field $\mathcal{B}(E)$. Denote by $B_b(E)$ the space of bounded measurable functions $f : E \to \mathbb{R}$ endowed with the supremum norm $|f|_{B_b(E)} = \sup_{x \in E} |f(x)|$. Denote by $\mathcal{M}(E)$ the Banach space of finite signed measures $\mu$ on $E$ endowed with the total variation norm $|\mu|_{\mathcal{M}(E)} = |\mu|(E)$, where $|\mu|$ is the total variation measure. Recall that if $\mu = f\, d\rho$, where $\rho \in \mathcal{M}_+(E)$ is a nonnegative measure and $f \in L^1(E, \rho)$, then $|\mu|_{\mathcal{M}(E)} = |f|_{L^1(E, \rho)}$. Denote by $\mathcal{P}(E) \subset \mathcal{M}(E)$ the set of probability measures on $E$. Moreover, we denote the Euclidean norm on $\mathbb{R}^{N}$ by $|\cdot|$ with inner product $\langle \cdot, \cdot \rangle$. Given some $A, B \in \mathbb{R}^{N \times N}$, we denote by $\lambda_{\min}(A)$ the minimum eigenvalue of $A$ and denote $A \succeq B$ if and only if $A - B$ is positive semidefinite.
\section{Technical Details}
\label{sec:technical}

The state-occupancy kernel $d^{\pi} \in \mathcal{P}(S|S)$ is defined by 
\begin{equation}
\label{eq:occupancy_s}
d^{\pi}(ds'|s)=(1-\gamma)\sum_{n=0}^{\infty}\gamma^nP^n_{\pi}(ds'|s)\,,
\end{equation}
where $P^n_{\pi}$ is the $n$-times product of the kernel $P_{\pi}$ with $P^0_{\pi}(ds'|s)\coloneqq \delta_s(ds')$. Moreover, for each $\pi \in \mathcal{P}(A|S)$ and $(s,a) \in S \times A$, we define the state-action occupancy kernel as
\begin{equation}
d^{\pi}(ds,da|s,a) = (1-\gamma) \sum_{n=0}^{\infty} \gamma^n (P^{\pi})^n(ds,da|s,a)
\end{equation} where $(P^{\pi})^{n}$ is the $n$-times product of the kernel $P^{\pi}$ with $(P^{\pi})^{0}(ds',da'|s,a) := \delta_{(s,a)}(ds',da')$. Given some initial state-action distribution $\beta \in \mathcal{P}(S\times A)$ with initial state distribution $\rho(ds) = \int_{A} \beta(da,ds)$, we define the state-occupancy and state-action occupancy measures as
\begin{equation}\label{eq:occupancy}
d^{\pi}_{\rho}(ds) = \int_{S} d^{\pi}(ds|s')\rho(ds'), \quad d^{\pi}_{\beta}(ds,da) = \int_{S \times A}d^{\pi}(ds,da|s',a') \beta(da',ds').
\end{equation}
Note that for all $E \in \mathcal{B}(S \times A)$, by defining the linear operator $J_{\pi} : \mathcal{P}(S \times A) \to  \mathcal{P}(S \times A)$ as
\begin{equation}\label{eq:transition_operator}
J_{\pi}\beta(E) = \int_{S\times A} P^{\pi}(E|s',a')\beta(ds',da'),
\end{equation}it directly holds that
\begin{equation}
d^{\pi}_{\beta}(da,ds) = (1-\gamma)\sum_{n=0}^{\infty} \gamma^n J_{\pi}^n\beta(da,ds),
\end{equation}
with $J_{\pi}^{n}$ the $n$-fold product of the operator $J_{\pi}$ with $J_{\pi}^{0}=I$, the identity operator on $\mathcal{P}(S \times A)$. 

A proof for the following dynamic programming principle for entropy regularised MDPs in general spaces can be found in \cite{david_fisher}.

\begin{theorem}[Dynamic Programming Principle]
\label{thm:dynamics_programming}
Let $\tau > 0$. The optimal value function $V^{*}_{\tau}$ is the unique bounded solution of the following Bellman equation:
\[
V^{\ast}_{\tau}(s)=-\tau\ln\int_{A}\exp\left(-
\frac{1}{\tau}Q^{\ast}_{\tau}(s,a)\right)\mu(da),
\]
where $Q^*_{\tau}\in B_b(S\times A)$ is defined by  
\[
Q^{*}_{\tau}(s,a)=c(s,a)+\gamma\int_S V_{\tau}^{*}(s')P(ds'|s,a)\,,
\quad \forall (s,a)\in S\times A\,.
\]Moreover, there is an optimal policy $\pi^* \in \mathcal{P}(A|S)$  given by
\[
\label{eq:optimal_policy}
\pi^*(da|s) = \exp\left(-\frac{1}{\tau }(Q^{\ast}_{\tau}(s,a)-V^{\ast}_{\tau}(s))\right)\mu(da)\,,
\quad \forall s\in S.
\]Finally, the value function $V^{\pi}_{\tau}$ is the unique bounded solution of the following Bellman equation for all $s \in S$
\[
V^{\pi}_{\tau}(s)=\int_{A}\left(Q_\tau^\pi(s,a)+\tau \ln \frac{d \pi}{d \mu}(a,s)\right)\pi(da|s)\,.
\]
\end{theorem}
The performance difference lemma, first introduced for tabular unregularised MDPs, has become fundamental in the analysis of MDPs as it acts a substitute for the strong convexity of the $\pi \mapsto V^{\pi}_{\tau}$ if the state-occupancy measure $d_{\rho}^{\pi}$ is ignored (e.g \cite{kakade_PD}, \cite{jordan}, \cite{lan}). By virtue of \cite{david_fisher}, we have the following performance difference for entropy regularised MDPs in Polish state and action spaces.
\begin{lemma}[Performance difference]
\label{lem:performance_diff}
For all $\rho \in \mathcal{P}(S)$ and $\pi,\pi'\in \Pi_{\mu}$, 
\begin{align*}
&V^{\pi}_\tau(\rho)-V^{\pi'}_\tau(\rho) \\
&\quad = \frac{1}{1-\gamma}\int_S \bigg[\int_A\left(Q^{\pi'}_{\tau}(s,a)+\tau \ln \frac{d \pi'}{d\mu}(a,s)\right)(\pi-\pi')(da|s) + \tau    \operatorname{KL}(\pi(\cdot | s)|\pi'(\cdot | s)) \bigg]d^{\pi}_\rho(ds)\,.
\end{align*}
\end{lemma}
The following three lemmas become useful in the stability analysis of both Algorithm \ref{algo:single_loop_ac} and Algorithm \ref{algo:double_loop_ac} aswell as the convergence analysis of Algorithm \ref{algo:double_loop_ac}.
\begin{lemma}\label{lemma:technical_KL_ln}
For all $\pi \in \Pi_{\mu}$ and $s \in S $ it holds that
\begin{equation}
\operatorname{KL}(\pi(\cdot|s) | \mu)  \leq 2\left|\log\frac{d\pi}{d\mu}(s,\cdot)- \int_{A}\log\frac{d\pi}{d\mu}(s,a) \mu(da) \right|_{B_b(A)}. 
\end{equation}
\end{lemma}
\begin{proof}By the definition and non-negativity of KL divergence it holds that
\begin{align}
\operatorname{KL}(\pi(\cdot|s) | \mu) &\leq \operatorname{KL}(\pi(\cdot|s) | \mu)  + \operatorname{KL}(\mu| \pi(\cdot|s)) \\
&= \int_{A} \log\frac{d\pi}{d\mu}(s,a)\pi(da|s) + \int_{A} \log\frac{d\mu}{d\pi}(s,a) \mu(da) \\
&= \int_{A} \log\frac{d\pi}{d\mu}(s,a)(\pi(da|s) - \mu(da)) \\
&= \int_{A} \left(\log\frac{d\pi}{d\mu}(s,a)- \int_{A}\log\frac{d\pi}{d\mu}(s,a) \mu(da)\right)(\pi(da|s) - \mu(da)) \\
&\leq \left|\log\frac{d\pi}{d\mu}(s,\cdot)- \int_{A}\log\frac{d\pi}{d\mu}(s,a) \mu(da) \right|_{B_b(A)}| \pi(\cdot|s) - \mu|_{\mathcal{M}(A)}\\
&\leq 2\left|\log\frac{d\pi}{d\mu}(s,\cdot)- \int_{A}\log\frac{d\pi}{d\mu}(s,a) \mu(da) \right|_{B_b(A)},
\end{align}where we used that $| \pi(\cdot|s) - \mu|_{\mathcal{M}(A)} \leq 2$ for all $s \in S$ in the final inequality.
\end{proof}
\begin{lemma}\label{lemma:sup_lemma}
For all $n\geq 0$, let $\{a_n \}_{n \geq 0}$ be a sequence of non-negative numbers such that
\begin{equation}
a_{n+1} \leq \kappa A_{n} + c
\end{equation} with $c \geq 0$, $0<\kappa < 1$ and $\mathrm{A}_n = \sup_{0 \leq r \leq n} a_r$. Then for all $n \in \mathbb{N}$ it holds that
\begin{equation}
\mathrm{A}_{n} \leq a_0 + \frac{c}{1-\kappa}.
\end{equation}
\end{lemma}
\begin{proof}
By definition, for any $n \in \mathbb{N}$ it holds that
\begin{equation}
A_{n+1}=\max\{A_n,a_{n+1}\}\leq \max\{A_n,\kappa A_n + c\}.
\end{equation}
If $\max\{A_{n},a_{n+1}\}=A_n$, then $A_{n+1}\leq A_0 + \frac{c}{1-\kappa} = a_0 + \frac{c}{1-\kappa}$ holds trivially for all $n \in \mathbb{N}$. On the other hand, we have
\begin{equation}
A_{n+1}\leq a_{n+1}\leq \kappa A_n + c.
\end{equation}
Iterating this inequality also yields $A_{n+1} \leq A_0 + \frac{c}{1-\kappa} = a_0 + \frac{c}{1-\kappa}$.
\end{proof}

\begin{lemma}\label{lemma:gradient}\cite{zorba2025convergenceactorcriticentropyregularised}
Let Assumption \ref{as:linearmdp} hold. Then for all $\theta \in \mathbb{R}^N$ and $\pi \in \Pi_{\mu}$ it holds that 
\begin{equation}
-\inner{g(\theta,\pi)}{\theta-\theta_{\pi}} \leq -(1-\sqrt{\gamma})(1-\gamma)\inner{\nabla_{\theta}{L}(\theta,\pi;\beta)}{\theta-\theta_{\pi}}
\end{equation} with \[\nabla_{\theta} L(\theta,\pi;\beta) = \int_{S \times A} (\inner{\theta}{\phi(s,a)} - Q^{\pi}_{\tau}(s,a))\phi(s,a)\beta(da,ds).\]
\end{lemma}

Lemma \ref{lemma:everyhing_bdd_if_KL_bdd} demonstrates that if there exists $R > 0$ such that $\operatorname{KL}(\pi^n(\cdot|s)|\mu) \leq R$ for all $s \in S$ and $n \in \mathbb{N}$, then the true state-action value function and critic parameters are also uniformly bounded. 

\begin{lemma}\label{lemma:everyhing_bdd_if_KL_bdd}
Suppose that $\operatorname{KL}(\pi^n(\cdot|s)|\mu) \leq R$ for all $n \in \mathbb{N}$ and $s \in S$.  Then it holds that
\begin{align}
&\left|Q^{\pi^n}\right|_{B_b(S\times A)} \leq \frac{1}{1-\gamma}\left(|c|_{B_b(S\times A)} + \tau \gamma R\right), \\
&    |\theta_{\pi^n}|_2 \leq \frac{1}{(1-\gamma)\lambda_{\beta}}\left(|c|_{B_b(S\times A)} + \tau \gamma R\right),
\end{align} for all $n \in \mathbb{N}$. Moreover suppose that the critic step size satisfies $0<h \leq \frac{\Gamma}{6(1+\gamma)^2}$, then there exists $\tilde{R} \geq 0$ independent of $h , \lambda > 0$ such that for all $n \geq 0$ it holds that 
\begin{align}
|\theta^{n+1}|_2^2 \leq \tilde{R}.
\end{align} 
\end{lemma}
\begin{proof}
Recall that for any $\pi \in \Pi_{\mu}$, $Q^{\pi}_{\tau}\in B_b(S\times A)$ is a fixed point of the Bellman operator defined in \eqref{eq:bellman_operator_Def}. Therefore it holds that
\begin{equation}
Q^{\pi}_{\tau}(s,a) = c(s,a) + \gamma \int_{S\times A}Q^{\pi}_{\tau}(s',a')P^{\pi}(ds',da'|s,a) + \tau\gamma \int_{S}\operatorname{KL}(\pi(\cdot|s')|\mu)P(ds'|s,a).
\end{equation}Taking the $|\cdot|_{B_b(S\times A)}$ on both sides and gives the result. Now let $\Sigma_{\beta} = \int_{S \times A} \phi(s,a) \phi(s,a)^{\top} \beta(da,ds)$. By Assumption \ref{as:linearmdp}, \ref{as:bounded_phi} and \ref{as:e_value} it holds that
\begin{align}
|\theta_{\pi}|_2 &= \left| \Sigma_{\beta} \Sigma_{\beta}^{-1} \theta_{\pi}\right|_2 \\
&\leq \frac{1}{\lambda_{\beta}}\left| Q^{\pi}_{\tau}\right|_{B_b(S\times A)}.
\end{align}Then using the first result yields the desired bound. 

Furthermore, suppose that $0<h<\frac{\Gamma}{3(1+\gamma)^2}$ Theorem \ref{lemma:theta_recursion_1}, it holds that
\begin{align}
|\theta^{n+1}|_2^2 &\leq |\theta_0|_2^2+ \frac{2\tau^2\gamma^2\left(3h + \frac{2}{\Gamma}\right)}{\Gamma - 3h(1+\gamma)^2}R^2 + \frac{|c|_{B_b(S\times A)}^2\left(3h + \frac{2}{\Gamma}\right)}{\Gamma - 3h(1+\gamma)^2} \\
&= \label{eq:with_alpha}|\theta_0|_2^2 + 2\tau^2\gamma^2R^2 \alpha(h) + |c|_{B_b(S\times A)}^2\alpha(h)
\end{align} where $\alpha(h) = \frac{3h+\frac{2}{\Gamma}}{\Gamma-3h(1+\gamma)^2}$ and observe that $\alpha(h)$ is non-decreasing and non-positive for all $h \in \left(0,\frac{\Gamma}{3(1+\gamma)^2}\right)$. Hence we further restrict the critic step size to
\begin{equation}
0<h\leq \bar{h}:= \frac{\Gamma}{6(1+\gamma)^2}.
\end{equation} Thus it holds that $\alpha(h) \leq \alpha(\bar{h})$ for all $h \in \left(0, \frac{\Gamma}{6(1+\gamma)^2} \right]$. Substituting this into \eqref{eq:with_alpha} we obtain that for all $n \in \mathbb{N}$ we have
\begin{equation}
|\theta^{n+1}|_2^2 \leq \bar{R}
\end{equation}for some $\bar{R} > 0$ independent of $h,\lambda > 0$.

\end{proof}
We now establish the main ingredients for the convergence proofs in Section \ref{sec:single_loop} and \ref{sec:double_loop}. 
One main ingredient of the proofs is the following Bregman proximal inequality. Let \[M_{\mu} = \left\{ m \in \mathcal{P}(A) | \log\frac{dm}{d\mu} \in B_b(A)\right\}\] and notice this is a convex subset of $\mathcal{P}(A)$. A proof of the following classical three point lemma/ identity can then be found in \cite{Korba22}.
\begin{lemma}[Three point lemma/ Bregman proximal inequality]
Let $G : M_{\mu} \to \mathbb{R}$ be convex. For all $m' \in M_{\mu}$ let
\begin{equation}
m^* = \argmin_{m\in M_{\mu}}\left\{G(m)+\operatorname{KL}(m|m') \right\}.
\end{equation}Then for all $m \in M_{\mu}$ we have
\begin{equation}
G(m) + \operatorname{KL}(m|m') \geq G(m^*) + \operatorname{KL}(m|m^*) + \operatorname{KL}(m^*|m')
\end{equation}
\end{lemma}
We also need the following crucial observation with a trivial proof.
\begin{lemma}\label{lemma:L_smooth}
Let $F:S \to \mathbb{R}$ be such that $F \leq 0$. Then for any $\pi \in \Pi_{\mu}$ and $s \in S$ it holds that
\begin{equation}
\frac{1}{1-\gamma}\int_{S} F(s')d_{s}^{\pi}(ds') \leq F(s)
\end{equation}
\end{lemma}
\begin{proof}
By definition of the state occupancy measure $d^{\pi}(\cdot|s) \in \mathcal{P}(S)$ in \ref{eq:occupancy_s} for any $s \in S$ and the fact that $P_\pi^0(ds'|s) = \delta_s(ds')$, we have for all $s \in S$ and $\pi \in \mathcal{P}(A|S)$ that
\begin{align}
\frac{1}{1-\gamma}\int_{S} F(s')d_{s}^{\pi}(s') &= \int_{S} F(s')P_{\pi}^0(ds'|s) + \sum_{k=1}^{\infty}\int_{S} \gamma^k F(s')P_{\pi}^k(ds'|s) \\
&\leq \int_{S}F(s')\delta_s(ds') = F(s).
\end{align}This concludes the proof.
\end{proof}
Theorem \ref{thm:main_up_to_errors} demonstrates sub-linear convergence of Algorithm \ref{algo:single_loop_ac} and \ref{algo:double_loop_ac} up to a cumulative sum of errors arising from approximating the critic.
\begin{theorem}\label{thm:main_up_to_errors}
For some $\theta^0 = \theta_0$ and $\pi^0 = \pi_0$, let $\left\{\pi^k, \theta^k \right\}_{k\geq 0}$ be the iterates of Algorithm \ref{algo:single_loop_ac} and Algorithm \ref{algo:double_loop_ac}. Suppose that $0 < \tau\lambda < 1$, then for all $\rho \in \mathcal{P}(S)$ and $n \in \mathbb{N}$ it holds that
\begin{align}\label{eq:main_up_to_errors_bound}
&\min_{0 \leq r \leq n-1} V^{\pi^r}_{\tau}(\rho) - V^{\pi^*}_{\tau}(\rho)\\
&\leq \frac{1}{\lambda(1-\gamma)n}\Bigg(\int_{S} \operatorname{KL}(\pi^* | \pi^{0})(s) d_{\rho}^{\pi^*}(ds) + \lambda\left(V^0(\rho) - V^*(\rho) \right) + \lambda c(\gamma)\sum_{k=0}^{n-1}|\theta^{k+1} - \theta_{\pi^k}|_2 \Bigg).
\end{align}
\end{theorem}

\begin{proof}

To ease notation let $V^{k} := V^{\pi^k}_{\tau}$ for $k \in \mathbb{N}$ and let $V^* := V^{\pi^*}_{\tau}$. Fix $s \in S$ and $\pi^k \in \Pi_{\mu}$. By the Three Point Lemma we have
\begin{align}
&\lambda \int_{A} A(s,a;\theta^{k+1})(\pi - \pi^k)(da|s) + \operatorname{KL}(\pi|\pi^k)(s) \nonumber\\
&\geq \lambda  \int_{A} A(s,a;\theta^{k+1})(\pi^{k+1} - \pi^k)(da|s) + \operatorname{KL}(\pi|\pi^{k+1})(s) + \operatorname{KL}(\pi^{k+1}|\pi^k)(s).
\end{align}
Rearranging, dividing through by $\lambda > 0$ and using that $0 < \tau\lambda < 1$ we have
\begin{align}\label{eq:carry_on}
&\frac{1}{\lambda}\left(\operatorname{KL}(\pi|\pi^{k+1})(s) - \operatorname{KL}(\pi|\pi^k)(s)\right) \nonumber\\
&\leq \int_{A} A(s,a;\theta^{k+1})(\pi - \pi^{k})(da|s) - \int_{A} A(s,a;\theta^{k+1})(\pi^{k+1} - \pi^k)(da|s) - \tau \operatorname{KL}(\pi^{k+1}|\pi^k)(s).
\end{align}
Now define
\begin{equation}
F(s) = \int_{A} A(s,a;\theta^{k+1})(\pi^{k+1} - \pi^k)(da|s) + \tau \operatorname{KL}(\pi^{k+1}|\pi^k)(s).
\end{equation}
Recall that by the policy mirror descent updates (as clarified in the calculations in \eqref{eq:mirror_descent_negative}), it holds that $F(s) \leq 0$ for all $s \in S$. Therefore, by the performance difference lemma and Lemma \ref{lemma:L_smooth} it holds that
\begin{align}
(V^{k+1} - V^{k})(s)
&= \frac{1}{1-\gamma}\int_{S}\left(\int_{A}A^{\pi^k}_{\tau}(s',a)(\pi^{k+1}-\pi^k)(da|s') +\tau \operatorname{KL}(\pi^{k+1}|\pi^k)(s')\right)d_{s}^{\pi^{k+1}}(ds') \nonumber\\
&= \frac{1}{1-\gamma}\int_{S}F(s')d_{s}^{\pi^{k+1}}(ds') + \frac{1}{1-\gamma}\int_{S}\left(A^{\pi^n}_{\tau}(s',a) -A(s',a;\theta^{n+1})\right)(\pi^{n+1} - \pi^{n})(da|s')d_{s}^{\pi^{k+1}}(ds') \\
&\leq \frac{1}{1-\gamma}\int_{S}F(s')d_{s}^{\pi^{k+1}}(ds') + \frac{2}{1-\gamma}|\theta^{k+1} - \theta_{\pi^k}|_2 \nonumber\\
&\leq F(s) + \frac{2}{1-\gamma}|\theta^{k+1} - \theta_{\pi^k}|_2 \label{eq:upper_bound_errors}
\end{align} where we added and subtracted the approximate advantage function in the second equality and used that for all $s \in S$ and $a \in A$, by  Hölder's inequality, Assumption \ref{as:linearmdp} and \ref{as:bounded_phi}, it holds that $A(s,a;\theta^{n+1}) - A^{\pi^n}_{\tau}(s,a) = Q(s,a;\theta^{n+1}) - Q^{\pi^n}_{\tau}(s,a) \leq \left|\theta^{n+1} - \theta_{\pi^n}\right|_2 |\phi(s,a)|_2 \leq \left|\theta^{n+1} - \theta_{\pi^n}\right|_2$.
Substituting this into \eqref{eq:carry_on}, for all $s \in S$ and $n \in \mathbb{N}$ it holds that
\begin{align}
&\frac{1}{\lambda}\left(\operatorname{KL}(\pi|\pi^{k+1})(s) - \operatorname{KL}(\pi|\pi^k)(s)\right) \nonumber\\
&\leq \int_{A} A(s,a;\theta^{k+1})(\pi - \pi^{k})(da|s) - \int_{A} A(s,a;\theta^{k+1})(\pi^{k+1} - \pi^k)(da|s) - \tau \operatorname{KL}(\pi^{k+1}|\pi^k)(s) \nonumber\\
&\leq \int_{A} A(s,a;\theta^{k+1})(\pi - \pi^{k})(da|s) - F(s)\nonumber\\
&\leq \int_{A} A(s,a;\theta^{k+1})(\pi - \pi^{k})(da|s) -\left(V^{k+1}(s) - V^{k}(s) \right) + \frac{2}{1-\gamma}|\theta^{k+1} - \theta_{\pi^k}|_2.
\end{align}
Now let $\pi = \pi^*$. Summing both sides over $k=0,1,\dots,n-1$ for any $n \in \mathbb{N}$ we have that
\begin{align}
&\operatorname{KL}(\pi^*|\pi^{n})(s) - \operatorname{KL}(\pi^*|\pi^{0})(s) \nonumber\\
&\leq \lambda\sum_{k=0}^{n-1}\int_{A} A(s,a;\theta^{k+1})(\pi^* - \pi^k)(da|s) - \lambda (V^{{n}} - V^{0})(s) + \frac{2\lambda}{1-\gamma}\sum_{k=0}^{n-1}|\theta^{k+1} - \theta_{\pi^k}|_2.
\end{align}
After multiplying through by $\lambda > 0$ and once again adding and subtracting the true advantage function, we have
\begin{align}
&\operatorname{KL}(\pi^*|\pi^{n})(s) - \operatorname{KL}(\pi^*|\pi^{0})(s) \nonumber\\
&\leq\lambda\sum_{k=0}^{n-1}\int_{A} A^{\pi^k}_{\tau}(s,a)(\pi^* - \pi^k)(da|s) - \lambda (V^{{n}} - V^{0})(s) \\
&\qquad + \lambda \sum_{k=0}^{n-1}\int_{A} \left(A(s,a;\theta^{n+1}) - A^{\pi^k}_{\tau}(s,a) \right)(\pi^* - \pi^k)(da|s) +\frac{2\lambda}{1-\gamma}\sum_{k=0}^{n-1}|\theta^{k+1} - \theta_{\pi^k}|_2.
\end{align}
Let $c(\gamma) := \max\left\{1,\frac{2}{1-\gamma} \right\}$. Again using that $A(s,a;\theta^{n+1}) - A^{\pi^n}_{\tau}(s,a) \leq \left|\theta^{n+1} - \theta_{\pi^n}\right|_2$ for all $s \in S$, $ a \in A$ and $n \in \mathbb{N}$ it holds that
\begin{align}
&\operatorname{KL}(\pi^*|\pi^{n})(s) - \operatorname{KL}(\pi^*|\pi^{0})(s) \nonumber\\
&\leq \lambda\sum_{k=0}^{n-1}\int_{A} A^{\pi^k}_{\tau}(s,a)(\pi^* - \pi^k)(da|s) - \lambda (V^{{n}} - V^{0})(s) + \lambda c(\gamma)\sum_{k=0}^{n-1}|\theta^{k+1} - \theta_{\pi^k}|_2.
\end{align} Moreover, using that $(V^{n}(s) - V^{0}(s)) \geq (V^{*}(s) - V^{0}(s))$, which holds for all $s \in S$ by the definition of $V^*$, and integrating over $d_{\rho}^{\pi^*}\in \mathcal{P}(S)$ we obtain
\begin{align}
&\int_{S} \operatorname{KL}(\pi^* | \pi^{n})(s) d_{\rho}^{\pi^*}(ds) - \int_{S} \operatorname{KL}(\pi^* | \pi^{0})(s) d_{\rho}^{\pi^*}(ds) \nonumber\\
&\leq \lambda\sum_{k=0}^{n-1}\int_{S}\int_{A} A^{\pi^k}_{\tau}(s,a)(\pi^* - \pi^k)(da|s) d_{\rho}^{\pi^*}(ds)
- \lambda (V^{*}(\rho) - V^{0}(\rho)) + \lambda c(\gamma)\sum_{k=0}^{n-1}|\theta^{k+1} - \theta_{\pi^k}|_2.
\end{align}
Finally, applying the performance difference lemma on the first term on the right hand side, we obtain
\begin{align}
&\int_{S} \operatorname{KL}(\pi^* | \pi^{n})(s) d_{\rho}^{\pi^*}(ds) - \int_{S} \operatorname{KL}(\pi^* | \pi^{0})(s) d_{\rho}^{\pi^*}(ds) \nonumber\\
&\leq \lambda(1-\gamma)\sum_{k=0}^{n-1}\left((V^* - V^k)(\rho) - \tau\int_{S} \operatorname{KL}(\pi^*|\pi^k)(s)d_{\rho}^{\pi^*}(ds) \right) -\lambda (V^{*}(\rho) - V^{0}(\rho)) \nonumber\\
&\qquad + \lambda c(\gamma)\sum_{k=0}^{n-1}|\theta^{k+1} - \theta_{\pi^k}|_2.
\end{align}
Rearranging, dropping negative terms from the right hand side and using that the minimum of a sequence is less than or equal to the average, for all $\rho \in \mathcal{P}(S)$ and $n \in \mathbb{N}$ it holds that

\begin{align}
&\min_{0 \leq r \leq n-1} V^{\pi^r}_{\tau}(\rho) - V^{\pi^*}_{\tau}(\rho)\\
&\leq \frac{1}{\lambda(1-\gamma)n}\Bigg(\int_{S} \operatorname{KL}(\pi^* | \pi^{0})(s) d_{\rho}^{\pi^*}(ds) + \lambda\left(V^0(\rho) - V^*(\rho) \right) + \lambda c(\gamma)\sum_{k=0}^{n-1}|\theta^{k+1} - \theta_{\pi^k}|_2 \Bigg).
\end{align}
which is exactly \eqref{eq:main_up_to_errors_bound} and thus concludes the proof.
\end{proof}

\section{Proofs of Section \ref{sec:single_loop}}
\subsection{Proof of Lemma \ref{lemma:theta_recursion_1}}\label{sec:proof_theta_recursion}
\begin{proof}
By definition of the critic updates in Algorithm \ref{algo:single_loop_ac}, for any $n \in \mathbb{N}$ it holds that 
\begin{align}\label{eq:lyaponov_theta}
|\theta^{n+1}|_2^2 &= |\theta^n - h g(\theta^n,\pi^n)|_2^2 \\
&= |\theta^n|_2^2 + h^2 |g(\theta^n,\pi^n)|_2^2 -2h \inner{\theta^n}{g(\theta^n,\pi^n)} \\
&= |\theta^n|_2^2 + h^2 I_1^n -2h I_2^n.
\end{align} where we let $I_1^n = |g(\theta^n,\pi^n)|_2^2$ and $I_2^n = \inner{\theta^n}{g(\theta^n,\pi^n)}$. Now let $\Gamma := \lambda_{\beta}(1-\gamma)(1-\sqrt{\gamma})$ and $\mathrm{K}_{n} = \sup_{s \in S} \operatorname{KL}(\pi^{n}(\cdot|s)|\mu)$. Lemma 5.2 of \cite{zorba2025convergenceactorcriticentropyregularised} shows that for any $\theta \in \mathbb{R}^{N}$ and $ \pi \in \Pi_{\mu}$, it holds that
\begin{align}
-\inner{g(\theta,\pi)}{\theta} &\leq -\frac{\Gamma}{2}|\theta|_2^2 + \frac{\tau^2\gamma^2}{\Gamma}\sup_{s \in S} \operatorname{KL}(\pi(\cdot|s)|\mu) + \frac{|c|_{B_b(S\times A)}^2}{\Gamma}.
\end{align}Therefore, for all $n \in \mathbb{N}$ we can upper bound $I_2^n$ as
\begin{align}
I_2^n = -2h\inner{g(\theta^n,\pi^n)}{\theta^n} &\leq 2h\left(-\frac{\Gamma}{2}|\theta^n|_2^2 + \frac{\tau^2\gamma^2}{\Gamma}\mathrm{K}_n^2 + \frac{|c|_{B_b(S\times A)}^2}{\Gamma} \right).
\end{align}Moreover, using the identity $(a + b + c)^2 \leq 3(a^2 + b^2 + c^2)$, a direct calculation shows that
\begin{align}
I_1^n = |g(\theta^n,\pi^n)|_2^2 \leq 3(1+\gamma)^2|\theta^n|_2^2 + 3|c|_{B_b(S\times A)}^2 + 3\tau^2\gamma^2 \mathrm{K}_n^2.
\end{align}Substituting these upper bounds into \eqref{eq:lyaponov_theta} it holds that
\begin{align}\label{eq:recursion_in_proof}
|\theta^{n+1}|_2^2 
&\leq \left(1+3h^2(1+\gamma)^2 - h\Gamma\right)|\theta^n|_2^2 
+ \left(3h^2\tau^2\gamma^2 + \frac{2h\tau^2\gamma^2}{\Gamma} \right)\mathrm{K}_n^2 
+ |c|_{B_b(S\times A)}^2\left(3h^2 + \frac{2h}{\Gamma} \right).
\end{align}
Now we choose $h > 0$ such that
\begin{align}
0<\left(1+3h^2(1+\gamma)^2 - h\Gamma\right) < 1,
\end{align}
which is achieved when $0<h < \frac{\Gamma}{3(1+\gamma)^2}$. Now let $\mathrm{M}_{n} := \sup_{0 \leq r \leq n} \mathrm{K}_{n}$, iterating \eqref{eq:recursion_in_proof} using that $\sum_{n=0}^{\infty} a^n = \frac{1}{1-a}$ for $0 < a < 1$, it holds that
\begin{align}\label{eq:sub_LN_in}
|\theta^{n+1}|_2^2 
&\leq |\theta_0|_2^2
+ \frac{\tau^2\gamma^2\left(3h + \frac{2}{\Gamma}\right)}{\Gamma - 3h(1+\gamma)^2}\mathrm{M}_n^2
+ \frac{|c|_{B_b(S\times A)}^2\left(3h + \frac{2}{\Gamma}\right)}{\Gamma - 3h(1+\gamma)^2},
\end{align}which concludes the proof.

\end{proof}

\subsection{Proof of Corollary \ref{cor:finite_action_space}}\label{sec:proof_corrolary} 
\begin{proof}
Since $|A| < \infty$ by assumption and $\mu\in \mathcal{P}(A)$ has full support on $A$, for all $s \in S$ and $n \in \mathbb{N}$ it holds that $\operatorname{KL}(\pi^{n}(\cdot|s)|\mu) \leq \log |A|$. Therefore, by Lemma \ref{lemma:everyhing_bdd_if_KL_bdd} 
\begin{align}
|\theta^{n+1}|_2^2
&\leq |\theta_0|_2^2 + \frac{\tau^2\gamma^2\left(3h + \frac{2}{\Gamma}\right)}{\Gamma - 3h(1+\gamma)^2}
\sup_{0 \leq r \leq n}\mathrm{K}_{r}^2  + \frac{|c|_{B_b(S\times A)}^2\left(3h + \frac{2}{\Gamma}\right)}{\Gamma - 3h(1+\gamma)^2} \\
&\leq |\theta_0|_2^2 + \frac{\tau^2\gamma^2\left(3h + \frac{2}{\Gamma}\right)}{\Gamma - 3h(1+\gamma)^2}\left(\log |A| \right)^2  + \frac{|c|_{B_b(S\times A)}^2\left(3h + \frac{2}{\Gamma}\right)}{\Gamma - 3h(1+\gamma)^2}.
\end{align}Finally, identically to Lemma \ref{lemma:everyhing_bdd_if_KL_bdd} we can make the right hand side independent of $h, \lambda >0$ by restricting $0 < h < \frac{\Gamma}{6(1+\gamma)^2}$.
\end{proof}

\subsection{Proof of Lemma \ref{lemma:log_recursion}}\label{sec:proof_log_recursion}

\begin{proof}

To ease the notation, through the analysis we drop the dependence on the policy in the approximate advantage function, i.e $A(s,a;\theta^{n+1},\pi^{n}) := A(s,a;\theta^{n+1})$. 
Then by the policy updates in Algorithm \ref{algo:single_loop_ac}, it holds that
\begin{equation}
\pi^{n+1}(\cdot|s) = \argmin_{m \in P(A)}\left\{\int_{A} A(s,a;\theta^{n+1})(m(da) - \pi^n(da|s)) + \frac{1}{\lambda}\operatorname{KL}(m|\pi^n(\cdot|s)) \right\}
\end{equation} for each $s \in S$. Lemma 1.4.3 of \cite{dupuis1997weak} shows that for each $s \in S$ the minimiser is achieved at
\begin{equation}
\log \frac{d\pi^{n+1}}{d\mu}(s,a)
=
\log \frac{d\pi^{n}}{d\mu}(s,a)
- \lambda\bigl(A(s,a;\theta^{n+1}) - \log(Z_n(s))\bigr)
\end{equation}

such that $Z_n(s) = \log \left(\int_{A}\exp\left(A(s,a';\theta^{n+1}) \right) \pi^{n}(da'|s)\right)$. Thus it holds that
\begin{align}
l_{n+1}(s,a) &= \log \frac{d\pi^{n}}{d\mu}(s,a) - \lambda(A(s,a;\theta^{n+1}) - \log Z_n(s)) \\
&\qquad - \left(\int_{A} \left(\log \frac{d\pi^{n}}{d\mu}(s,a') - \lambda(A(s,a';\theta^{n+1}) - \log Z_n(s)) \right)\mu(da') \right).
\end{align}Cancelling out the normalisation constants $Z_n(s)$ for each $s \in S$ (as they are independent of $a \in A$) and using the definition of the approximate advantage function \eqref{eq:approximate_advantage}, we have
\begin{align}
l_{n+1}(s,a) &= l_{n}(s,a)
-\lambda\Bigg(Q(s,a;\theta^{n+1}) - \int_{A}Q(s,a';\theta^{n+1})\mu(da') + \tau l_n(s,a) \Bigg).
\end{align}After grouping like terms we arrive at
\begin{align}
l_{n+1}(s,a) = (1-\tau\lambda)l_n(s,a) -\lambda \left( Q(s,a;\theta^{n+1}) - \int_{A}Q(s,a';\theta^{n+1})\mu(da') \right).
\end{align} Taking the $|\cdot|_{B_b(S\times A)}$ norm on both sides and using the triangle inequality on the right hand side yields
\begin{align}
\left| l_{n+1} \right|_{B_b(S\times A)} 
\leq (1-\tau\lambda) \left| l_{n}\right|_{B_b(S\times A)} +2\lambda\left| Q(\cdot,\cdot;\theta^{n+1})\right|_{B_b(S\times A)}.
\end{align}

\end{proof}

\subsection{Proof of Theorem \ref{thm:KL_BDD_small_gamma}}\label{sec:proof_KL_bdd_small_gamma}
\begin{proof}
By Lemma \ref{lemma:log_recursion}, for all $n \in \mathbb{N}$ it holds that
\begin{align}
|l_{n+1}|_{B_b(S\times A)} &\leq (1-\tau\lambda)|l_{n}|_{B_b(S\times A)} + 2\lambda\left| Q(\cdot,\cdot;\theta^{n+1})\right|_{B_b(S\times A)} \\
&\leq (1-\tau\lambda)|l_{n}|_{B_b(S\times A)} + 2\lambda |\theta^n|_2
\end{align}where we used that $\left| Q(\cdot,\cdot;\theta^{n+1})\right|_{B_b(S\times A)} \leq |\phi(s,a)|_2 |\theta^n|_2 \leq |\theta^n|_2$ which holds by Hölder's inequality and Assumption \ref{as:bounded_phi}. Moreover, by assumption we have $ 0 < \tau\lambda < 1$, thus iterating this recursion we have
\begin{align}
\left| l_{n+1} \right|_{B_b(S\times A)}
&\leq (1-\tau\lambda)^{n+1}\left| l_0\right|_{B_b(S\times A)} + 2\lambda \sum_{k=0}^{n} (1-\tau\lambda)^{n-k}|\theta^{k+1}|_2. \\
&\leq \label{eq:bounded_finite_actions}C_1 + 2\lambda \sum_{k=0}^{n} (1-\tau\lambda)^{\,n-k}\,|\theta^{k+1}|_2,
\end{align}
where $C_1 >0$ is such that $\left|l_0\right|_{B_b(S\times A)} \leq C_1$, which holds without loss of generality since $\pi_0 \in \Pi_{\mu}$.
Now squaring both sides, using the identity $(a+b)^2 \leq 2a^2 + 2b^2$ and Hölder's inequality we have
\begin{align}
\left| l_{n+1} \right|_{B_b(S\times A)}^2 
&\leq  C_1^2 
+ 4\lambda^2 \left(\sum_{k=0}^{n} (1-\tau\lambda)^{n-k}|\theta^{k+1}|_2\right)^2 \nonumber \\
&\leq C_1^2+ \frac{4\lambda}{\tau} \sum_{k=0}^n (1-\tau\lambda)^{n-k}|\theta^{k+1}|_2^2. \label{eq:recursion_theta_log}
\end{align}

Now we turn to Lemma \ref{lemma:theta_recursion_1}, where given that $0<h < \frac{\Gamma}{3(1+\gamma)^2}$ for all $n \in \mathbb{N}$ it holds that
\begin{align}\label{eq:sub_LN_in}
|\theta^{n+1}|_2^2 
&\leq |\theta_0|_2^2
+ \frac{\tau^2\gamma^2\left(3h + \frac{2}{\Gamma}\right)}{\Gamma - 3h(1+\gamma)^2}\mathrm{M}_n^2
+ \frac{|c|_{B_b(S\times A)}^2\left(3h + \frac{2}{\Gamma}\right)}{\Gamma - 3h(1+\gamma)^2}.
\end{align} with $\mathrm{M}_{n} := \sup_{0 \leq r \leq n} \mathrm{K}_{n}$ and $\mathrm{K}_n := \sup_{s \in S} \operatorname{KL}(\pi^n(\cdot|s)|\mu)$. By Lemma \ref{lemma:technical_KL_ln} it holds that for all $\operatorname{KL}(\pi^{n}(\cdot|s) | \mu) \leq 2 |l_{n}|_{B_b(S \times A)}$ for all $s \in S$ and $ n \in \mathbb{N}$. Thus it also holds that $\mathrm{M}_{n}^2 \leq 4\mathrm{L}_n^2$ for all $n \in \mathbb{N}$ and hence we have
\begin{align}
|\theta^{n+1}|_2^2 
&\leq |\theta_0|_2^2
+ \frac{4\tau^2\gamma^2\left(3h + \frac{2}{\Gamma}\right)}{\Gamma - 3h(1+\gamma)^2}\mathrm{L}_n^2
+ \frac{|c|_{B_b(S\times A)}^2\left(3h + \frac{2}{\Gamma}\right)}{\Gamma - 3h(1+\gamma)^2}.
\end{align}

Substituting this into \eqref{eq:recursion_theta_log} and using that $\sum_{m=0}^{n} (1-\tau \lambda)^{m} \leq \frac{1}{\tau\lambda}$ for any $n \in \mathbb{N}$, a direct calculation shows that
\begin{equation}\label{eq:recursion_before_lemma}
|l_{n+1}|_{B_b(S\times A)}^2 \leq \kappa(h) \mathrm{L}_{n}^2 + C_2(h)
\end{equation} such that
\begin{equation}
\kappa(h) := \frac{16\gamma^2(3h + \frac{2}{\Gamma})}{\Gamma -3h(1+\gamma)^2}, \quad C_2(h) := C_1^2+\frac{4}{\tau^2}\,|\theta_0|_2^2+\frac{4\,|c|_{B_b(S\times A)}^2\left(3h + \frac{2}{\Gamma}\right)}{\tau^2\left(\Gamma - 3h(1+\gamma)^2\right)}.
\end{equation} A direct calculation shows that given that \begin{equation}
0 < h < \min\left\{
\frac{\Gamma}{3(1+\gamma)^2}\,,
\;
\frac{\Gamma^2 - 32\gamma^2}{\Gamma\bigl(48\gamma^2 + 3(1+\gamma)^2\bigr)}
\right\},
\end{equation} then the conditions of Lemma \ref{lemma:theta_recursion_1} and $0<\kappa(h) < 1$ both hold, provided that $\frac{32\gamma^2}{\Gamma^2} < 1$. Therefore, by Lemma \ref{lemma:sup_lemma} for all $n \geq 0$ it holds that
\begin{align}
\mathrm{L}_{n}^2 &\leq \mathrm{L}_0^2 + \frac{C_2(h)}{1-\kappa(h)}.
\end{align} Furthermore, by Lemma \ref{lemma:technical_KL_ln} it holds that $\operatorname{KL}(\pi^{n}(\cdot|s)|\mu) \leq \mathrm{K}_{n} \leq 2\mathrm{L}_{n}$ for all $s \in S$, therefore we have
\begin{align}
\operatorname{KL}(\pi^{n}(\cdot|s)|\mu)^2 &\leq 4\mathrm{L}_0^2 + \frac{4C_2(h)}{1-\kappa(h)}.
\end{align}
Finally, we further restrict the critic step size by defining
\begin{equation}\label{eq:h_def}
\bar h := \frac{1}{2}\min\left\{
\frac{\Gamma}{3(1+\gamma)^2}\,,
\frac{\Gamma^2 - 32\gamma^2}{\Gamma\bigl(48\gamma^2 + 3(1+\gamma)^2\bigr)}
\right\},
\end{equation}
and restricting $0<h\leq \bar h$. Since $0<\kappa(h)<1$ on $(0,\bar h]$, we have $1-\kappa(h)\geq 1-\kappa(\bar h)>0$. Moreover, both $\kappa(h)$ and $C_2(h)$ are non-decreasing in $h$ on $(0,\bar h]$, hence
\[
\kappa(h)\leq \kappa(\bar h), \qquad C_2(h)\leq C_2(\bar h), \qquad 0<h\leq \bar h.
\]
Therefore, for all $n\geq 0$ and all $s\in S$,
\begin{align}
\operatorname{KL}(\pi^{n}(\cdot|s)|\mu)^2 
&\leq 4\mathrm{L}_0^2 + \frac{4C_2(h)}{1-\kappa(h)} \leq 4\mathrm{L}_0^2 + \frac{4C_2(\bar h)}{1-\kappa(\bar h)}:= R,
\end{align}
where $R$ depends only on $\gamma,\tau,\pi_0,|\theta_0|_2,|c|_{B_b(S\times A)}, \lambda_{\beta}$ and is independent of $h > 0$.

\end{proof}
\subsection{Proof of Theorem \ref{thm:lipschitz_holder}}
\label{sec:proof_lipschitz}
\begin{proof}
Recall that $Q^{\pi} \in B_{b}(S \times A)$ is a fixed point of the Bellman operator defined in \eqref{eq:bellman_operator_Def} for all $\pi \in \Pi_{\mu}$. Hence it holds that
\begin{align}
Q^{\pi}(s,a) - Q^{\pi'}(s,a) &= \gamma \int_{S\times A}Q^{\pi}(s',a')\pi(da'|s')P(ds'|s,a) \\
&\qquad- \gamma\int_{S\times A}Q^{\pi'}(s',a')\pi'(da'|s')P(ds'|s,a) \\
&\qquad + \tau \gamma \left(\int_{S} \left(\operatorname{KL}(\pi|\mu)(s') - \operatorname{KL}(\pi'|\mu)(s') \right)P(ds'|s,a) \right).
\end{align}
Now we plus and minus $\gamma \int_{S\times A}Q^{\pi'}(s',a')\pi(da'|s')P(ds'|s,a)$ on the right hand side for each $s \in S$ and $a \in A$ to arrive at
\begin{align}
Q^{\pi}(s,a) - Q^{\pi'}(s,a) &= \gamma \int_{S\times A}\left(Q^{\pi}(s',a') - Q^{\pi'}(s',a')\right)\pi(da'|s')P(ds'|s,a) \\
&\qquad + \gamma \int_{S \times A} Q^{\pi'}(s',a') (\pi - \pi')(da'|s') \\
&\qquad + \tau \gamma \left(\int_{S} \left(\operatorname{KL}(\pi|\mu)(s') - \operatorname{KL}(\pi'|\mu)(s') \right)P(ds'|s,a) \right).
\end{align}
Rearranging and taking the $|\cdot|_{B_b(S\times A)}$ norm on both sides and using the triangle inequality we have

\begin{align}\label{eq:difference_Q_functions}
\left|Q^{\pi} - Q^{\pi'} \right|_{B_b(S \times A)} &\leq \frac{\gamma}{1-\gamma} \left| Q^{\pi'}\right|_{B_b(S\times A)} \sup_{s \in S} \left|\pi(\cdot|s) - \pi'(\cdot|s) \right|_{\mathcal{M}(A)} \\
&\qquad +\frac{\tau \gamma}{1-\gamma} \left| \operatorname{KL}(\pi|\mu)- \operatorname{KL}(\pi'|\mu) \right|_{B_b(S \times A)}
\end{align}We shift our focus to last term on the right hand side. By definition for any $\pi \in \Pi_{\mu}$ it holds that
\begin{align}
&\left|\operatorname{KL}(\pi|\mu) - \operatorname{KL}(\pi'|\mu)\right|_{B_b(S)}  \\
&= \left| \int_{A} \log \frac{d\pi}{d\mu}(\cdot,a) \pi(da|\cdot) - \int_{A} \log \frac{d\pi'}{d\mu}(\cdot,a) \pi'(da|\cdot)\right|_{B_b(S)} \\
&= \left| \int_{A} \log \left(\frac{d\pi}{d\pi'}(\cdot,a) \frac{d\pi'}{d\mu}(\cdot,a)\right) \pi(da|\cdot) - \int_{A} \log \frac{d\pi'}{d\mu}(\cdot,a) \pi'(da|\cdot)\right|_{B_b(S)} \\
&= \left| \int_{A} \log \frac{d\pi}{d\pi'}(\cdot,a)\pi(da|\cdot) + \int_{A} \log \frac{d\pi'}{d\mu}(\cdot,a)\pi(da|\cdot) - \int_{A} \log \frac{d\pi'}{d\mu}(\cdot,a) \pi'(da|\cdot)\right|_{B_b(S)} \\
&\leq \sup_{s \in S}\operatorname{KL}(\pi|\pi')(s) + \left| \int_{A} \log \frac{d\pi'}{d\mu}(\cdot,a)\left( \pi - \pi'\right)(da|\cdot)\right|_{B_b(S)} \\
&= \sup_{s \in S}\operatorname{KL}(\pi|\pi')(s) + \left| \int_{A} \left(\log \frac{d\pi'}{d\mu}(\cdot,a) - \int_{A}\log \frac{d\pi'}{d\mu}(\cdot,a)\mu(da)\right)\left(\pi - \pi'\right)(da|\cdot)\right|_{B_b(S)}.
\end{align}Now let $\pi = \pi^{n+1}$ and $\pi'=\pi^{n}$. By Pinsker's inequality, for all $n \in \mathbb{N}$ it holds that
\begin{align}
\left|\operatorname{KL}(\pi^{n+1}|\mu) - \operatorname{KL}(\pi^{n}|\mu)\right|_{B_b(S)}&\leq \sup_{s \in S} \operatorname{KL}(\pi^{n+1}|\pi^n)(s) + \left| l_{n}\right|_{B_b(S\times A)} \sup_{s \in S}\left| \pi^{n+1}(\cdot|s) - \pi^{n}(\cdot|s) \right|_{\mathcal{M}(A)} \\ 
&\leq \sup_{s \in S}\operatorname{KL}(\pi^{n+1}|\pi^n)(s) + \sqrt{2}\left|l_n\right|_{B_b(S\times A)} \sup_{s \in S}\operatorname{KL}(\pi^{n+1}|\pi^n)(s)^\frac{1}{2}.
\end{align}
Substituting into \eqref{eq:difference_Q_functions} and using Pinsker's inequality once again we arrive at
\begin{align}
\left|Q^{\pi^{n+1}}_{\tau} - Q^{\pi^{n}}_{\tau} \right|_{B_b(S \times A)} &\leq  \frac{\gamma}{1-\gamma} \left| Q^{\pi^{n}}_{\tau}\right|_{B_b(S\times A)} \sup_{s \in S}\operatorname{KL}(\pi^{n+1}|\pi^{n})(s)^{\frac{1}{2}} \\
&\qquad + \frac{\tau \gamma}{1-\gamma} \sup_{s \in S} \left( \operatorname{KL}(\pi^{n+1}|\pi^{n})(s) + \sqrt{2}\left| l_n\right|_{B_b(S\times A)} \operatorname{KL}(\pi^{n+1}|\pi^n)(s)^\frac{1}{2} \right).
\end{align} After grouping terms we have
\begin{align}
\left|Q^{\pi^{n+1}}_{\tau} - Q^{\pi^{n}}_{\tau} \right|_{B_b(S \times A)} &\leq  \left( \frac{\sqrt{2}\tau\gamma  \left| l_n\right|_{B_b(S\times A)}}{1-\gamma}+ \frac{\gamma \left| Q^{\pi^{n}}_{\tau}\right|_{B_b(S\times A)}}{1-\gamma} \right) \sup_{s \in S} \operatorname{KL}(\pi^{n+1}|\pi^n)(s)^\frac{1}{2}  \\&\qquad+ \frac{\tau \gamma}{1-\gamma}\sup_{s \in S} \operatorname{KL}(\pi^{n+1}|\pi^n)(s).
\end{align} Since by assumption there exists $R \geq 0$ such that for any $n \in \mathbb{N}$ and $s \in S$ we have $ |\theta^n|_2 + \operatorname{KL}(\pi^n(\cdot|s)|\mu) \leq R$, by Lemma \ref{lemma:everyhing_bdd_if_KL_bdd} there exists $\alpha_1 \geq 0$ such that $ \frac{\sqrt{2}\tau\gamma  \left| l_n\right|_{B_b(S\times A)}}{1-\gamma}+ \frac{\gamma \left| Q^{\pi^{n}}_{\tau}\right|_{B_b(S\times A)}}{1-\gamma} \leq \alpha_1$ for all $n \in \mathbb{N}$. Hence letting $\alpha_2 = \frac{\tau\gamma}{1-\gamma}$ we conclude the proof.
\end{proof}

\subsection{Proof of Theorem \ref{thm:single_loop_conv}}
Before we address the proof of Theorem \ref{thm:single_loop_conv} we must first establish the following regularity property of the policies produced by Algorithm \ref{algo:single_loop_ac}.

\begin{lemma}\label{lemma:consecutive_KL_bounded}
For some $\theta^0 = \theta_0 \in \mathbb{R}^{N}$ and $\pi^0 = \pi_0 \in \Pi_{\mu}$ let $\{\theta^n,\pi^n\}_{n \in \mathbb{N}}$ be the iterates of Algorithm \ref{algo:single_loop_ac}. Let $0<\lambda\tau < 1$. Then for all $n \in \mathbb{N}$ and $s \in S$, it holds that
\begin{equation}
\operatorname{KL}(\pi^{n+1}(\cdot|s)|\pi^{n}(\cdot|s)) \leq \frac{\lambda}{1-\lambda\tau}|\theta^{n+1}|_2
\end{equation}
\end{lemma}
\begin{proof}
By definition, for each $s \in S$ the mirror descent updates can be applied pointwise through
\begin{equation}
\pi^{n+1}(\cdot|s) = \argmin_{m \in P(A)}\left\{\int_{A} A(s,a;\theta^{n+1})(m(da) - \pi^n(da|s)) + \frac{1}{\lambda}\operatorname{KL}(m|\pi^n(\cdot|s)) \right\}.
\end{equation} Therefore for all $s \in S$ it holds that
\begin{equation}
\int_{A} A(s,a;\theta^{n+1})\pi^{n+1}(da|s) + \frac{1}{\lambda}\operatorname{KL}(\pi^{n+1}(\cdot|s)|\pi^n(\cdot|s)) \leq 0.
\end{equation}Rearranging and using the definition of the approximate advantage function, for each $s \in S$ it holds that
\begin{align}
\frac{1}{\lambda}\operatorname{KL}(\pi^{n+1}(\cdot|s)|\pi^n(\cdot|s)) &\leq -\int_{A} A(s,a;\theta^{n+1})\pi^{n+1}(da|s) \\
&= - \int_{A} \left(Q(s,a;\theta^{n+1}) + \tau\log \frac{d\pi^n}{d\mu}(s,a) \right)\pi^{n+1}(da|s) \\
&\leq |\theta^{n+1}|_2 - \tau \int_{A} \log \left( \frac{d\pi^n}{d\pi^{n+1}}(s,a)\frac{d\pi^{n+1}}{d\mu}(s,a)\right)\pi^{n+1}(da|s) \\
&= |\theta^{n+1}|_2 - \tau \int_{A} \left(\log  \frac{d\pi^n}{d\pi^{n+1}}(s,a) + \log \frac{d\pi^{n+1}}{d\mu}(s,a)\right)\pi^{n+1}(da|s)\\
&= |\theta^{n+1}|_2 + \tau \operatorname{KL}(\pi^{n+1}(\cdot|s) |\pi^{n}(\cdot|s)) - \tau\operatorname{KL}(\pi^{n+1}(\cdot|s) |\mu) \\
&\leq |\theta^{n+1}|_2 + \tau \operatorname{KL}(\pi^{n+1}(\cdot|s) |\pi^{n}(\cdot|s)).
\end{align} where we used the non-negativity of KL divergence in the final inequality.
After rearranging we conclude the proof.
\end{proof}

We are now ready to prove Theorem \ref{thm:single_loop_conv}.
\begin{proof}
Theorem \ref{thm:main_up_to_errors} states that for all $n \in \mathbb{N}$ and $\rho \in \mathcal{P}(S)$, we have
\begin{align}\label{eq:main_thm_sub_in}
&\min_{0 \leq r \leq n-1} V^{\pi^r}_{\tau}(\rho) - V^{\pi^*}_{\tau}(\rho)\\
&\leq \frac{1}{\lambda(1-\gamma)n}\Bigg(\int_{S} \operatorname{KL}(\pi^* | \pi^{0})(s) d_{\rho}^{\pi^*}(ds) + \lambda\left(V^0(\rho) - V^*(\rho) \right) + \lambda c(\gamma)\sum_{k=0}^{n-1}|\theta^{k+1} - \theta_{\pi^k}|_2 \Bigg).
\end{align}
with $c(\gamma) = \max\left\{1,\frac{2}{1-\gamma} \right\}$. Then by definition of the temporal difference updates in Algorithm \ref{algo:single_loop_ac}, for all $n \in \mathbb{N}$ it holds that
\begin{align}
\left|\theta^{n+1} - \theta_{\pi^n}\right|_2^2 &= \left|\theta^n -hg(\theta^n,\pi^n) - \theta_{\pi^n} \right|_2^2 \\
&= |\theta^n - \theta_{\pi^n}|_2^2 +h^2 |g(\theta^n,\pi^n)|_2^2 -2h\inner{\theta^{n} - \theta_{\pi^n}}{g(\theta^n,\pi^n)} \\
&\leq |\theta^n - \theta_{\pi^n}|_2^2 +2h^2 (1+\gamma)|\theta^{n} - \theta_{\pi^n}|_{2}^2 -2h\inner{\theta^{n} - \theta_{\pi^n}}{g(\theta^n,\pi^n)} \\
&\leq |\theta^n - \theta_{\pi^n}|_2^2 +2h^2 (1+\gamma)|\theta^{n} - \theta_{\pi^n}|_{2}^2  \\ &\qquad -2h(1-\gamma)(1-\sqrt{\gamma})\inner{\nabla_{\theta}L(\theta^n,\pi^n)}{\theta^n - \theta_{\pi^n}} \\
&\leq  |\theta^n - \theta_{\pi^n}|_2^2 +2h^2 (1+\gamma)|\theta^{n} - \theta_{\pi^n}|_{2}^2  - 2h\Gamma|\theta^n - \theta_{\pi^n}|_2^2 \\
&= (1-h\Gamma)|\theta^n - \theta_{\pi^n}|_2^2 + (2(1+\gamma)h^2 - h\Gamma)|\theta^n - \theta_{\pi^n}|_2^2.
\end{align}where we used Lemma \ref{lemma: bound_g_squared} in the first inequality, Lemma \ref{lemma:gradient} in the second inequality and the strong convexity of $\theta \mapsto L(\theta,\pi;\beta)$ in the final inequality.  Now choosing $0 < h < \frac{\Gamma}{2(1+\gamma)}$ it holds that $2(1+\gamma)^2h^2 - h\Gamma \leq 0$ and thus we arrive at
\begin{equation}
\left|\theta^{n+1} - \theta_{\pi^n}\right|_2^2 \leq (1-h\Gamma) \left|\theta^{n} - \theta_{\pi^n}\right|_2^2.
\end{equation}Now we add and subtract $\theta_{\pi^{n-1}}$ and use Young's inequality with some $\epsilon > 0$ to arrive at
\begin{align}
\left|\theta^{n+1} - \theta_{\pi^n}\right|_2^2&\leq \left(1+ \epsilon\right)(1-h\Gamma)|\theta^n - \theta_{\pi^{n-1}}|_2^2 + \left(1+\frac{1}{\epsilon}\right)(1-h\Gamma)|\theta_{\pi^{n}} - \theta_{\pi^{n-1}}|_2^2.
\end{align} Choosing $\epsilon > 0$ such that $(1+\epsilon)(1-h\Gamma) = 1-\frac{h\Gamma}{2}$, we obtain $\epsilon = \frac{h\Gamma}{2(1-h\Gamma)}$ which also implies that $\left(1+\frac{1}{\epsilon}\right) = \frac{2}{h\Gamma} - 1 > 0$. Hence for all $n \in \mathbb{N}$ it holds that
\begin{align}\label{eq:before_subs_stepsizes}
\left|\theta^{n+1} - \theta_{\pi^n}\right|_2^2&\leq \left(1-\frac{h\Gamma}{2}\right)|\theta^n - \theta_{\pi^{n-1}}|_2^2 + \left(\frac{2}{h\Gamma}-1 \right)(1-h\Gamma)|\theta_{\pi^{n}} - \theta_{\pi^{n-1}}|_2^2 \\
&\leq \left(1-\frac{h\Gamma}{2}\right)|\theta^n - \theta_{\pi^{n-1}}|_2^2 + \frac{2}{h\Gamma}|\theta_{\pi^{n}} - \theta_{\pi^{n-1}}|_2^2.
\end{align} Now for each $n \in \mathbb{N}$, we seek to control the errors coming from the moving target $|\theta_{\pi^{n}} - \theta_{\pi^{n-1}}|_2^2$. To this end, Lemma \ref{thm:lipschitz_holder} shows that for some $\alpha_1,\alpha_2>0$ independent of $h,\lambda > 0$, for all $n \in \mathbb{N}$ it holds that
\begin{align}
\left|Q^{\pi^n}_{\tau} - Q^{\pi^{n-1}}_{\tau}\right|_{B_b(S \times A)}
&\leq \alpha_1 \sup_{s \in S} \operatorname{KL}(\pi^{n} \mid \pi^{n-1})(s)
+ \alpha_2 \sup_{s \in S} \operatorname{KL}(\pi^{n} \mid \pi^{n-1})(s)^2.
\end{align}
Then Lemma \ref{lemma:consecutive_KL_bounded} states that for each $s \in S$ and $n \in \mathbb{N}$ it holds that
\begin{equation}
\operatorname{KL}(\pi^{n}|\pi^{n-1})(s) \leq \frac{\lambda}{1-\tau\lambda}|\theta^{n-1}|_2 \leq \frac{\lambda}{1-\tau\lambda}R,
\end{equation}where used that the critic parameters are uniformly bounded by assumption. Finally observe that for each $n \in \mathbb{N}$ it holds that $|\theta_{\pi^{n}} - \theta_{\pi^{n-1}}|_2\leq \frac{1}{\lambda_{\beta}}\left|Q^{\pi^n}_{\tau} - Q^{\pi^{n-1}}_{\tau}\right|_{B_b(S \times A)}$ where $\lambda_{\beta} > 0$ by Assumption \ref{as:e_value}. 
Therefore, there exists $\alpha_3 > 0$ independent of $h,\lambda >0$ such that for all $n \in \mathbb{N}$ we have
\begin{align}\label{eq:lipschitz_upper_bound}
\left|\theta_{\pi^n} - \theta_{\pi^{n-1}}\right|_2^2
&\leq \alpha_3\left(\frac{\lambda}{1-\tau\lambda} + \frac{\lambda^2}{(1-\tau\lambda)^2} \right).
\end{align}
Substituting \eqref{eq:lipschitz_upper_bound} into \eqref{eq:before_subs_stepsizes} it holds that

\begin{align}
\left|\theta^{n+1} - \theta_{\pi^n}\right|_2^2
&\leq \left(1-\frac{h\Gamma}{2}\right)^n\left|\theta^{1} - \theta_{\pi^{0}}\right|_2^2
+ \frac{2\alpha_3}{h\Gamma}\left(\frac{\lambda}{1-\tau\lambda} + \frac{\lambda^2}{(1-\tau\lambda)^2}\right)
\sum_{k=0}^{n-1}\left(1-\frac{h\Gamma}{2}\right)^k \\
&= \left(1-\frac{h\Gamma}{2}\right)^n\left|\theta^{1} - \theta_{\pi^{0}}\right|_2^2
+ \frac{2\alpha_3}{h\Gamma}\left(\frac{\lambda}{1-\tau\lambda} + \frac{\lambda^2}{(1-\tau\lambda)^2}\right)
\frac{1-\left(1-\frac{h\Gamma}{2}\right)^n}{\frac{h\Gamma}{2}} \\
&\leq \left(1-\frac{h\Gamma}{2}\right)^n\left|\theta^{1} - \theta_{\pi^{0}}\right|_2^2
+ \frac{4\alpha_3}{h^2\Gamma^2}\left(\frac{\lambda}{1-\tau\lambda} + \frac{\lambda^2}{(1-\tau\lambda)^2}\right).
\end{align}
Taking the average over $0\leq k \leq n-1$, a direct calculation shows that

\begin{align}
\frac{1}{n}\sum_{k=0}^{n-1}\left|\theta^{k+1}-\theta_{\pi^k}\right|_2^2\leq \frac{2}{n h\Gamma}\left|\theta^{1}-\theta_{\pi^{0}}\right|_2^2
+ \frac{4\alpha_3}{h^2\Gamma^2}\left(\frac{\lambda}{1-\tau\lambda} + \frac{\lambda^2}{(1-\tau\lambda)^2}\right).
\end{align} By Hölder's inequality and the fact that $0 < \lambda\tau < 1$, it also holds that
\begin{align}
\frac{1}{n}\sum_{k=0}^{n}|\theta^{k+1} - \theta_{\pi^k}|_2 &\leq \left(\frac{1}{n}\sum_{k=0}^{n}|\theta^{k+1} - \theta_{\pi^k}|_2^2 \right)^{\frac{1}{2}} \\
&\leq \bigg(\frac{2}{n h\Gamma}\left|\theta^{1}-\theta_{\pi^{0}}\right|_2^2
+ \frac{4\alpha_3}{h^2\Gamma^2}\left(\frac{\lambda}{1-\tau\lambda} + \frac{\lambda^2}{(1-\tau\lambda)^2}\right) \bigg)^{\frac{1}{2}} \\
&\leq C \bigg(\frac{1}{nh} + \frac{\lambda^2}{h^2} \bigg)^{\frac{1}{2}} \\
&\leq C\left(\frac{1}{n^{\frac{1}{2}}h^{\frac{1}{2}}} + \frac{\lambda}{h}\right).
\end{align}with $C^2 := \max\left\{\frac{2}{\Gamma}\left|\theta^{1}-\theta_{\pi^{0}}\right|_2^2,\;
\frac{4\alpha_3}{\Gamma^2}\left(\frac{1}{\lambda(1-\tau\lambda)}+\frac{1}{(1-\tau\lambda)^2}\right)\right\}$. Substituting this into \eqref{eq:main_thm_sub_in}, we conclude with
\begin{align}\label{eq:main_thm_sub_in}
&\min_{0 \leq r \leq n-1} V^{\pi^r}_{\tau}(\rho) - V^{\pi^*}_{\tau}(\rho)\\
&\leq \frac{1}{\lambda(1-\gamma)n}\Bigg(\int_{S} \operatorname{KL}(\pi^* | \pi^{0})(s) d_{\rho}^{\pi^*}(ds) + \lambda\left(V^0(\rho) - V^*(\rho) \right) + \lambda c(\gamma)\left(\frac{C}{n^{\frac{1}{2}}h^{\frac{1}{2}}} + \frac{\lambda}{h} \right) \Bigg).
\end{align}

\end{proof}

\section{Proofs of Section \ref{sec:double_loop}}

\subsection{Proof of Lemma \ref{lemma: bound_g_squared}}\label{sec:proof_gradient_squared_bound}
\begin{proof}
Firstly observe by adding and subtracting $Q^{\pi}_{\tau}\in B_{b}(S \times A)$ and using that $\mathrm{T}^{\pi}_{\tau} Q^{\pi}_{\tau} = Q^{\pi}_{\tau}$ for all $\pi \in \Pi_{\mu}$ where $\mathrm{T}^{\pi}_{\tau}: B_{b}(S\times A) \to B_{b}(S\times A)$ is the Bellman operator defined in \eqref{eq:bellman_operator_Def}, we can express the semi-gradient as
\begin{align}
g(\theta,\pi) &= \int_{S \times A} \left(Q(s,a;\theta) -\mathrm{T}^{\pi}_{\tau}Q(s,a;\theta) \right) \phi(s,a)d_{\beta}^{\pi}(ds,da)\\
&= \int_{S \times A} \left(Q(s,a;\theta) -Q^{\pi}_{\tau}(s,a) -\left( \mathrm{T}^{\pi}_{\tau}Q(s,a;\theta) - \mathrm{T}^{\pi}_{\tau}Q^{\pi}_{\tau}(s,a)\right) \right) \phi(s,a)d_{\beta}^{\pi}(ds,da)
\end{align}for all $\theta \in \mathbb{R}^N$ and $\pi \in \Pi_{\mu}$. To ease notation, for each $s \in S$ and $a \in A$ we define $\varepsilon(s,a) = Q(s,a;\theta) - Q^{\pi}_{\tau}(s,a)$. Expanding the Bellman operator, a direction calculations shows that 
\begin{equation}
g(\theta,\pi) = \int_{S \times A} \left(\varepsilon(s,a) - \gamma\int_{S \times A} \varepsilon(s',a')P^{\pi}(ds',da'|s,a)\right)\phi(s,a) d_{\beta}^{\pi}(ds,da). 
\end{equation} Moreover, by using Young's and Hölder's inequalities along with Assumption \ref{as:bounded_phi}, it holds that
\begin{align}
\left| g(\theta,\pi) \right|_2^2 &\leq  \int_{S \times A} \left(\varepsilon(s,a) - \gamma \int_{S \times A} \varepsilon(s',a')P^{\pi}(ds',da'|s,a)\right)^2d_{\beta}^{\pi}(ds,da) \\
&\leq 2\int_{S \times A } \varepsilon(s,a)^2 d_{\beta}^{\pi}(da,ds) + 2\gamma^2 \int_{S \times A \times S \times A} \varepsilon(s',a')^2P^{\pi}(ds',da'|s,a)
d_{\beta}^{\pi}(ds,da)\\
&=2\int_{S \times A } \varepsilon(s,a)^2 d_{\beta}^{\pi}(da,ds) + 2\gamma^2 \int_{S \times A} \varepsilon(s',a')^2 d_{J_{\pi}\beta}^{\pi}(ds',da'),
\end{align}where $J_{\pi} : \mathcal{P}(S \times A) \to \mathcal{P}(S \times A)$ is the one step transition operator defined in \eqref{eq:transition_operator}. By Lemma 5.1 of \cite{zorba2025convergenceactorcriticentropyregularised}, it holds that $d_{J_{\pi}\beta}^{\pi}(E) \leq \frac{1}{\gamma} d_{\beta}^{\pi}(E)$ for all $E \in \mathcal{B}(S \times A).$ Therefore using the non-negativity of the integrand, for all $\theta \in \mathbb{R}^{N}$ and $\pi \in \Pi_{\mu}$ it holds that 
\begin{align}
\left| g(\theta,\pi) \right|_2^2 &\leq 2\int_{S \times A } \varepsilon(s,a)^2 d_{\beta}^{\pi}(da,ds) + 2\gamma \int_{S \times A} \varepsilon(s,a)^2 d^{\pi}_{\beta}(ds,da) \\
&= 2\left(1+ \gamma\right)\int_{S \times A } \varepsilon(s,a)^2 d_{\beta}^{\pi}(da,ds) \\
&\leq 2(1+\gamma)|\theta - \theta_{\pi}|_2^2,
\end{align}where the last inequality follows from the fact that for all $s \in S$ and $a \in A$ we have $\varepsilon^2(s,a) = \left(Q(s,a;\theta) - Q^{\pi}_{\tau}(s,a) \right)^2 = \left(\inner{\theta - \theta_{\pi}}{\phi(s,a)} \right)^2 \leq |\theta - \theta_{\pi}|_2^2 |\phi(s,a)|_2^2$ from Hölder's inequality and Assumption \ref{as:linearmdp} and \ref{as:bounded_phi}.

\end{proof}

\subsection{Proof of Theorem \ref{thm:inner_convergence}}\label{sec:proof_inner_conv}
\begin{proof}
Fix $n \in \mathbb{N}$ and let $M:= M(n)$ to ease notation. By definition of the critic updates in Algorithm \ref{algo:double_loop_ac}, for all $n \in \mathbb{N}$ it holds that
\begin{align}
|\theta^{n+1} - \theta_{\pi^n}|_2^2 & = |\theta^{n,M} - \theta_{\pi^n}|_2^2 \\
&=|\theta^{n,M-1} - h g(\theta^{n,M-1},\pi^n) - \theta_{\pi^n}|_2^2 \\
&=|\theta^{n,M-1} - \theta_{\pi^n}|_2^2 + h^2\left| g(\theta^{n,M-1},\pi^n) \right|_2^2 -2h \inner{\theta^{n,M-1} - \theta_{\pi^n}}{g(\theta^{n,M-1},\pi^n) }.
\end{align}
By Lemma \ref{lemma: bound_g_squared}, it holds that $\left|g(\theta,\pi) \right|_2^2 \leq 2(1+\gamma)\left| \theta - \theta_{\pi}\right|_2^2$ for all $\pi \in \Pi_{\mu}$ and $\theta \in \mathbb{R}^N$. Hence it holds that
\begin{align}
|\theta^{n+1} - \theta_{\pi^n}|_2^2 &\leq|\theta^{n,M-1} - \theta_{\pi^n}|_2^2 + 2h^2(1+\gamma)|\theta^{n,M-1} - \theta_{\pi^n}|_2^2 -2h \inner{\theta^{n,M-1} - \theta_{\pi^n}}{g(\theta^{n,M-1},\pi^n) }.
\end{align}
Using Lemma \ref{lemma:gradient} on the final term we arrive at
\begin{align}
|\theta^{n,M} - \theta_{\pi^n}|_2^2 &\leq |\theta^{n,M-1} - \theta_{\pi^n}|_2^2  + 2h^2(1+\gamma)|\theta^{n,M-1} - \theta_{\pi^n}|_2^2 \\
&\qquad- 2h(1-\gamma)(1-\sqrt{\gamma})\inner{\nabla_{\theta} L(\theta^{n,M-1},\pi^{n};\beta)}{\theta^{n,M-1} - \theta_{\pi^n}}.
\end{align}
We now proceed by using standard tools from convex analysis. That is, using the $\lambda_{\beta}$-strong convexity of $L(\cdot,\pi;\beta)$ and Assumption \ref{as:linearmdp}, we have
\begin{align} 
|\theta^{n,M} - \theta_{\pi^n}|_2^2 &\leq |\theta^{n,M-1} - \theta_{\pi^n}|_2^2  + 2h^2(1+\gamma)|\theta^{n,M-1} - \theta_{\pi^n}|_2^2 \\
&\qquad- 2h(1-\gamma)(1-\sqrt{\gamma})\inner{\nabla_{\theta} L(\theta^{n,M-1},\pi^{n};\beta)}{\theta^{n,M-1} - \theta_{\pi^n}} \\
&\leq |\theta^{n,M-1} - \theta_{\pi^n}|_2^2  + 2h^2(1+\gamma)|\theta^{n,M-1} - \theta_{\pi^n}|_2^2 \\
&\qquad- 2h(1-\gamma)(1-\sqrt{\gamma})\lambda_{\beta}|\theta^{n,M-1} - \theta_{\pi^n}|_2^2 \label{eq:before_strong_convexity}\\
&= \left(1 -2h\Gamma + 2(1+\gamma)h^2\right)|\theta^{n,M-1} - \theta_{\pi^n}|_2^2.
\end{align}
with $\Gamma = \lambda_{\beta}(1-\gamma)(1-\sqrt{\gamma})$. 
Choosing the critic step size $h$ such that 
\begin{equation}
0 < h < \min\left\{\frac{\Gamma}{2(1+\gamma)},\frac{1}{\Gamma}\right\},
\end{equation}
it holds that $2(1+\gamma)h^2 \leq h\Gamma$ and $h\Gamma \in (0,1)$, therefore
\begin{align}
|\theta^{n,M} - \theta_{\pi^n}|_2^2 &\leq (1 -h\Gamma)|\theta^{n,M-1} - \theta_{\pi^n}|_2^2.
\end{align}For each fixed $n \in \mathbb{N}$, iterating over the $M > 0$ critic steps yields
\begin{align}
|\theta^{n,M} - \theta_{\pi^n}|_2^2 &\leq (1-h\Gamma)^{M}|\theta^{n} - \theta_{\pi^n}|_2^2.
\end{align}
We conclude by using the standard identity $(1-x)^{M} \leq e^{-x M}$ for $x \in (0,1)$ and $M > 0$ to arrive at
\begin{equation}
|\theta^{n+1} - \theta_{\pi^n}|_2^2 \leq e^{-M h\Gamma}|\theta^{n} - \theta_{\pi^n}|_2^2.
\end{equation}

\end{proof}

\subsection{Proof of Lemma \ref{lemma:value_bound}}\label{sec:proof_value_bound}
\begin{proof}
By the performance difference lemma, for all $s \in S$ and $n \in \mathbb{N}$ it holds that
\begin{align}
V^{\pi^{n+1}}_{\tau}(s) - V^{\pi^{n}}_{\tau}(s)
&= \frac{1}{1-\gamma}\int_{S}\int_{A} \left(A^{\pi^n}_{\tau}(s',a)(\pi^{n+1} - \pi^n)(da|s') + \tau \operatorname{KL}(\pi^{n+1} | \pi^n)(s') \right)d_{s}^{\pi^{n+1}}(ds') \nonumber\\
&\leq \frac{1}{1-\gamma}\int_{S}\int_{A} \left(A^{\pi^n}_{\tau}(s',a)(\pi^{n+1} - \pi^n)(da|s') + \frac{1}{\lambda}\operatorname{KL}(\pi^{n+1} | \pi^n)(s') \right)d_{s}^{\pi^{n+1}}(ds') \nonumber\\
&= \frac{1}{1-\gamma}\int_{S}\int_{A} \left(A(s',a;\theta^{n+1})(\pi^{n+1} - \pi^n)(da|s') + \frac{1}{\lambda}\operatorname{KL}(\pi^{n+1} | \pi^n)(s')\right)d_{s}^{\pi^{n+1}}(ds')\\
&\qquad+ \frac{1}{1-\gamma}\int_{S}\int_{A}\left(A^{\pi^n}_{\tau}(s',a) - A(s',a;\theta^{n+1}) \right)(\pi^{n+1} - \pi^n)(da|s')d_{s}^{\pi^{n+1}}(ds')\\
&\leq \frac{1}{1-\gamma}\int_{S}\int_{A} \left(A(s',a;\theta^{n+1})(\pi^{n+1} - \pi^n)(da|s') + \frac{1}{\lambda}\operatorname{KL}(\pi^{n+1} | \pi^n)(s')\right)d_{s}^{\pi^{n+1}}(ds') \label{eq:from_performance_diff}\\
&\qquad+ \frac{2}{(1-\gamma)}|\theta^{n+1} - \theta_{\pi^n}|_{2}. \nonumber
\end{align}where we used that $\tau \leq \frac{1}{\lambda}$ in the first inequality, added and subtracted the approximate advantage function in the second equality and used the fact for all $s \in S$ and $a \in A$, by  Hölder's inequality, Assumption \ref{as:linearmdp} and \ref{as:bounded_phi}, it holds that $A(s,a;\theta^{n+1}) - A^{\pi^n}_{\tau}(s,a) = Q(s,a;\theta^{n+1}) - Q^{\pi^n}_{\tau}(s,a) \leq \left|\theta^{n+1} - \theta_{\pi^n}\right|_2 |\phi(s,a)|_2 \leq \left|\theta^{n+1} - \theta_{\pi^n}\right|_2$. Now recall the policy mirror descent update in Algorithm \ref{algo:double_loop_ac}
\begin{equation}\label{eq:mirror_descent_negative}
\pi^{n+1}(\cdot|s) = \argmin_{m \in P(A)}\left\{\int_{A} A(s,a;\theta^{n+1})(m(da) - \pi^n(da|s)) + \frac{1}{\lambda}\operatorname{KL}(m|\pi^n(\cdot|s)) \right\}.
\end{equation}
Since we obtain the minimum at $\pi^{n+1}$, for any $s \in S$ and $\pi(\cdot|s) \in \mathcal{P}(A)$ it holds that
\begin{equation}
\int_{A} A(s,a;\theta^{n+1})(\pi^{n+1} - \pi^n)(da|s) + \frac{1}{\lambda}\operatorname{KL}(\pi^{n+1}|\pi^{n})(s) \leq     \int_{A} A(s,a;\theta^{n+1})(\pi - \pi^n)(da|s) + \frac{1}{\lambda}\operatorname{KL}(\pi|\pi^{n})(s).
\end{equation}Choosing $\pi = \pi^{n+1}$ yields
\begin{equation}
\int_{A} A(s,a;\theta^{n+1})(\pi^{n+1} - \pi^n)(da|s) + \frac{1}{\lambda}\operatorname{KL}(\pi^{n+1}|\pi^{n})(s) \leq0 .
\end{equation}
Substituting this into \eqref{eq:from_performance_diff} and using the definition of $\pi^* \in \Pi_{\mu}$, it hence holds that
\begin{equation}\label{eq:approximate_decreasing_one_step}
V^{\pi^*}_{\tau}(s)\leq V^{\pi^{n+1}}_{\tau}(s) \leq V^{\pi^{n}}_{\tau}(s) + \frac{2}{1-\gamma}|\theta^{n+1} - \theta_{\pi^n}|_{2}
\end{equation}
for all $s \in S$. Finally, applying Lemma \ref{thm:inner_convergence} we arrive at 
\begin{equation}
V^{\pi^*}_{\tau}(s)\leq V^{\pi^{n+1}}_{\tau}(s) \leq V^{\pi^{n}}_{\tau}(s) + \frac{2e^{-\frac{M(n)h\Gamma}{2}}}{1-\gamma}|\theta^{n} - \theta_{\pi^n}|_{2},
\end{equation}which concludes the proof.

\end{proof}

\subsection{Proof of Theorem \ref{thm:bounded_under_log_growth}}\label{sec:proof_bounded_under_log_growth}
\begin{proof}
Recall that for each $n \in \mathbb{N}$ and $s \in S$, $a \in A$, the normalised log densities are defined as $l_n(s,a) = \ln \frac{d\pi^n}{d\mu}(s,a) - \int_{A} \ln \frac{d\pi^n}{d\mu}(s,a')\mu(da')$. Then by Lemma \ref{lemma:log_recursion}, for all $n \in \mathbb{N}$ it holds that
\begin{align}
|l_{n+1}|_{B_b(S\times A)} &\leq (1-\tau\lambda)|l_{n}|_{B_b(S\times A)} +  2\lambda \left| Q(\cdot,\cdot;\theta^{n+1})\right|_{B_b(S\times A)} \\
&\leq(1-\tau\lambda)|l_{n}|_{B_b(S\times A)}+  2\lambda \left(\left| Q^{\pi^n}_{\tau}\right|_{B_b(S\times A)} + \left|Q(\cdot,\cdot;\theta^{n+1}) -Q^{\pi^n}_{\tau}\right|_{B_b(S \times A)} \right) \\
&\leq (1-\tau\lambda)|l_{n}|_{B_b(S\times A)}+  2\lambda \left(\left| Q^{\pi^n}_{\tau}\right|_{B_b(S\times A)} + |\theta^{n+1} - \theta_{\pi^n}|_2\right) \\
&\leq (1-\tau\lambda)\mathrm{L}_n+  2\lambda \left(\left| Q^{\pi^n}_{\tau}\right|_{B_b(S\times A)} + e^{-\frac{M(n)h\Gamma}{2}}|\theta^{n} - \theta_{\pi^n}|_2\right). \label{eq:recursion_2}
\end{align} 
where we added and subtracted $Q_{\tau}^{\pi^n}$ in the second inequality and used Hölder's inequality along with Assumption \ref{as:bounded_phi} in the third inequality. In the final inequality we used Theorem \ref{thm:inner_convergence} and used that by definition $|l_n|_{B_b(S\times A)} \leq \sup_{0\leq r\leq n} |l_r|_{B_b(S\times A)} :=\mathrm{L}_n$. Now to upper bound $Q^{\pi^n}_{\tau} \in B_b(S \times A)$, by definition it holds that
\begin{align}
Q^{\pi^n}_{\tau}(s,a) = c(s,a) + \gamma \int_{S} V^{\pi^n}_{\tau}(s')P(ds'|s,a),
\end{align} for all $s \in S$ and $a \in A$. Taking the $| \cdot|_{B_b(S\times A)}$ and rearranging, it holds that
\begin{align}\label{eq:Q_function_bound}
\left| Q^{\pi^n}_{\tau}\right|_{B_b(S\times A)} &\leq |c|_{B_b(S\times A)} + \gamma \left| V^{\pi^n}_{\tau}\right|_{B_b(S)}.
\end{align}
Now by Lemma \ref{lemma:value_bound}, for all $n \in \mathbb{N}$ we have
\begin{equation}\label{eq:to_iterate}
V^{\pi^*}_{\tau}(s)\leq V^{\pi^{n+1}}_{\tau}(s) \leq V^{\pi^{n}}_{\tau}(s) + \frac{2e^{-\frac{M(n)h\Gamma}{2}}}{1-\gamma}|\theta^{n} - \theta_{\pi^n}|_{2}.
\end{equation}
Recursively applying the upper bound \eqref{eq:to_iterate}, for all $s \in S$ and $n \in \mathbb{N}$ it holds that
\begin{equation}
V^{\pi^*}_{\tau}(s) \leq V^{\pi^n}_{\tau}(s) \leq V^{\pi^0}_{\tau}(s) + \frac{2}{1-\gamma}\sum_{k=0}^{n-1}e^{-\frac{M(k)h\Gamma}{2}}|\theta^{k} - \theta_{\pi^{k}}|_2.
\end{equation}
Taking the $|\cdot|_{B_b(S)}$ norm on both sides it also holds that
\begin{align}
\left| V^{\pi^n}_{\tau} \right|_{B_b(S)} &\leq \max\left\{\left|V^{\pi^0}_{\tau}\right|_{B_b(S)} + \frac{2}{1-\gamma}\sum_{k=0}^{n-1}e^{-\frac{M(k)h\Gamma}{2}}|\theta^{k} - \theta_{\pi^{k}}|_2,  \left| V^{\pi^*}_{\tau} \right|_{B_b(S)}\right\} \\
&\leq \left|V^{\pi^0}_{\tau}\right|_{B_b(S)} + \left| V^{\pi^*}_{\tau} \right|_{B_b(S)}  + \frac{2}{1-\gamma}\sum_{k=0}^{n-1}e^{-\frac{M(k)h\Gamma}{2}}|\theta^{k} - \theta_{\pi^{k}}|_2.
\end{align}
Substituting this into \eqref{eq:Q_function_bound}, for all $n\in \mathbb{N}$ it holds that
\begin{align}
\left| Q^{\pi^n}_{\tau}\right|_{B_b(S\times A)} &\leq |c|_{B_b(S\times A)} + \gamma \left| V^{\pi^n}_{\tau}\right|_{B_b(S)} \\
&\leq |c|_{B_b(S\times A)} +\label{eq:Q_bound}\gamma\left|V^{\pi^0}_{\tau}\right|_{B_b(S)} + \gamma \left| V^{\pi^*}_{\tau} \right|_{B_b(S)} + \frac{2\gamma}{1-\gamma} \sum_{k=0}^{n-1}e^{-\frac{M(k)h \Gamma}{2}}|\theta^{k} - \theta_{\pi^{k}}|_2. 
\end{align}
Let $\alpha_1 := |c|_{B_b(S\times A)} + \gamma\left|V^{\pi^0}_{\tau}\right|_{B_b(S)} + \gamma \left| V^{\pi^*}_{\tau} \right|_{B_b(S)}$ and $c(\gamma) = \max\left\{1,\frac{2\gamma}{1-\gamma} \right\}$. Substituting \eqref{eq:Q_bound} into \eqref{eq:recursion_2}, we have
\begin{align}\label{eq:recursion_almost_there}
|l_{n+1}|_{B_b(S\times A)} &\leq (1-\tau\lambda)\mathrm{L}_n + 2\lambda \left( c(\gamma)\sum_{k=0}^{n}e^{-\frac{M(k)h\Gamma}{2}}|\theta^{k} - \theta_{\pi^{k}}|_2 +  \alpha_1\right).
\end{align} 

By the triangle inequality it holds that $|\theta^{k} - \theta_{\pi^{k}}|_2 \leq |\theta^k|_2 + |\theta_{\pi^k}|_2$ for each $k \in \mathbb{N}$.
Then by Lemma \ref{lemma:theta_recursion_1} for all $k\in \mathbb{N}$ it holds that
\begin{align}
|\theta^{k}|_2^2 
&\leq |\theta_0|_2^2+ \frac{\tau^2\gamma^2\left(3h + \frac{2}{\Gamma}\right)}{\Gamma - 3h(1+\gamma)^2}\sup_{0\leq r \leq k}\mathrm{K}_{r}^2 + \frac{|c|_{B_b(S\times A)}^2\left(3h + \frac{2}{\Gamma}\right)}{\Gamma - 3h(1+\gamma)^2} \\
&\leq |\theta_0|_2^2+ \frac{4\tau^2\gamma^2\left(3h + \frac{2}{\Gamma}\right)}{\Gamma - 3h(1+\gamma)^2}\mathrm{L}_{k}^2 + \frac{|c|_{B_b(S\times A)}^2\left(3h + \frac{2}{\Gamma}\right)}{\Gamma - 3h(1+\gamma)^2} \\
&\leq \label{eq:theta_bound_in_proof}|\theta_0|_2^2+ \frac{4\tau^2\gamma^2\left(3h + \frac{2}{\Gamma}\right)}{\Gamma - 3h(1+\gamma)^2}\mathrm{L}_{n}^2 + \frac{|c|_{B_b(S\times A)}^2\left(3h + \frac{2}{\Gamma}\right)}{\Gamma - 3h(1+\gamma)^2},
\end{align}
where we used Lemma \ref{lemma:technical_KL_ln} in the second inequality and the fact that $\mathrm{L}_k \leq \mathrm{L}_n$ for all $k\leq n$. Similarly by Lemma \ref{lemma:everyhing_bdd_if_KL_bdd} it holds that
\begin{align}
|\theta_{\pi^k}|_2 
&\leq \frac{|c|_{B_b(S\times A)}}{(1-\gamma)\lambda_{\beta}} + \frac{2\tau\gamma}{(1-\gamma)\lambda_{\beta}}\mathrm{L}_k \\
&\leq \label{theta_pi_bound}\frac{|c|_{B_b(S\times A)}}{(1-\gamma)\lambda_{\beta}} + \frac{2\tau\gamma}{(1-\gamma)\lambda_{\beta}}\mathrm{L}_n,
\end{align} 
where we again used that $\mathrm{L}_k\leq \mathrm{L}_n$ for $k\leq n$. Therefore for all $k \in \left\{ 0, \ldots, n \right\}$ it holds that
\begin{align}\label{eq:after_triangle}
|\theta^{k} - \theta_{\pi^{k}}|_2
&\leq  (\delta_1 + \delta_2\,\mathrm{L}_n),
\end{align}
with
\begin{equation}
\delta_1
= \left(|\theta_0|_2^2 + \frac{|c|_{B_b(S\times A)}^2\left(3h + \frac{2}{\Gamma}\right)}{\Gamma - 3h(1+\gamma)^2}\right)^{\frac{1}{2}} + \frac{|c|_{B_b(S\times A)}}{(1-\gamma)\lambda_{\beta}}, \quad     \delta_2
= 2\tau\gamma\left(\frac{\left(3h + \frac{2}{\Gamma}\right)}{\Gamma - 3h(1+\gamma)^2}\right)^{\frac{1}{2}} + \frac{2\tau\gamma}{(1-\gamma)\lambda_{\beta}}.
\end{equation}

Hence substituting \eqref{eq:after_triangle} into \eqref{eq:recursion_almost_there}, we arrive at
\begin{align}\label{eq:recursion_simplified}
|l_{n+1}|_{B_b(S\times A)} &\leq (1-\tau\lambda)\mathrm{L}_n + 2\lambda \left( c(\gamma)\sum_{k=0}^{n}e^{-\frac{M(k)h\Gamma}{2}}|\theta^{k} - \theta_{\pi^{k}}|_2 +  \alpha_1\right) \\
&\leq \beta_1(n) \mathrm{L}_n + \beta_2(n),
\end{align}
where
\begin{align}
\beta_1(n) := (1-\tau\lambda) + 2\lambda c(\gamma)\delta_2\sum_{k=0}^{n}e^{-\frac{M(k) h \Gamma}{2}},
\end{align}
\begin{align}
\beta_2(n) := 2\lambda c(\gamma)\delta_1\sum_{k=0}^{n}e^{-\frac{M(k) h \Gamma}{2}} + 2\lambda\alpha_1.
\end{align}
For all $n \in \mathbb{N}$ we choose $M(n)$ large enough such that
\begin{align}\label{eq:beta_target}
\beta_1(n) \leq 1-\frac{\tau\lambda}{2}.
\end{align}
A direct calculation shows that this is satisfied when
\begin{align}
\sum_{k=0}^{n}e^{-\frac{M(k) h \Gamma}{2}} \leq \frac{\tau}{4c(\gamma)\delta_2}.
\end{align}Therefore for all $n \in \mathbb{N}$ let $M(n) \geq \frac{4}{h\Gamma}\log(c(n+1))$ for some $c > 0$. This implies that for all $n \in \mathbb{N}$
\begin{equation}
\sum_{k=0}^{n}e^{-\frac{M(k) h \Gamma}{2}} \leq \frac{1}{c^2}\sum_{k=0}^{n} \frac{1}{(k+1)^2} \leq \frac{2}{c^2} \leq \frac{\tau}{4c(\gamma)\delta_2},
\end{equation} which yields $c^2 \geq \frac{8c(\gamma)\delta_2}{\tau}$. Therefore with choice of $M(n)$ for all $n \in \mathbb{N}$, \eqref{eq:recursion_simplified} becomes
\begin{align}\label{eq:l_bound_final}
|l_{n+1}|_{B_b(S\times A)} &\leq \beta_1(n) \mathrm{L}_n + \beta_2(n)\\
&\leq \left(1-\frac{\tau\lambda}{2}\right)\mathrm{L}_n + \alpha_3.
\end{align}
A direct application of Lemma \ref{lemma:sup_lemma} yields that for all $n \in \mathbb{N}$
\begin{align}
\mathrm{L}_n &\leq \mathrm{L}_0 + \frac{\alpha_3}{1-\left(1-\frac{\tau\lambda}{2}\right)} \\
&= \mathrm{L}_0 + \frac{2\alpha_3}{\tau\lambda}.
\end{align}Finally, after applying Lemma \ref{lemma:technical_KL_ln} and Lemma \ref{lemma:everyhing_bdd_if_KL_bdd}, we conclude the proof.
\end{proof}

\section{Proofs of Section \ref{sec:convergence_double_loop}}

\subsection{Proof of Theorem \ref{thm:linear_convergence_log_growth}}\label{sec:proof_linear_conv_log_growth}
\begin{proof}
Theorem \ref{thm:main_up_to_errors} states that for all $n \in \mathbb{N}$ and $\rho \in \mathcal{P}(S)$, we have
\begin{align}\label{eq:main_thm_sub_in}
&\min_{0 \leq r \leq n-1} V^{\pi^r}_{\tau}(\rho) - V^{\pi^*}_{\tau}(\rho)\\
&\leq \frac{1}{\lambda(1-\gamma)n}\Bigg(\int_{S} \operatorname{KL}(\pi^* | \pi^{0})(s) d_{\rho}^{\pi^*}(ds) + \lambda\left(V^0(\rho) - V^*(\rho) \right) + \lambda c(\gamma)\sum_{k=0}^{n-1}|\theta^{k+1} - \theta_{\pi^k}|_2 \Bigg).
\end{align}

with $c(\gamma) = \max\left\{1,\frac{2}{1-\gamma} \right\}$. We are thus left with controlling the average of the critic errors.
To that end, by Theorem \ref{thm:inner_convergence}, for any $n \in \mathbb{N}$ it holds that
\begin{equation}
\bigl|\theta^{n+1} - \theta_{\pi^n}\bigr|_2
\leq e^{-\frac{M(n)h\Gamma}{2}}\bigl|\theta^{n} - \theta_{\pi^n}\bigr|_2.
\end{equation}
Now by assumption we have that for each $n \in \mathbb{N}$, $M := M(n) \geq \frac{4}{h\Gamma}\log\!\big(c(n+1)\big)$ with $c>0$ the constant from Theorem \ref{thm:bounded_under_log_growth}. Thus, by Theorem
\ref{thm:bounded_under_log_growth} and Lemma \ref{lemma:everyhing_bdd_if_KL_bdd}, there exists $R>0$ such that
$\bigl|\theta^{k} - \theta_{\pi^k}\bigr|_2 \leq R$ for all $k \geq 0$. Therefore for any $n \in \mathbb{N}$ we have
\begin{align}
\bigl|\theta^{n+1} - \theta_{\pi^n}\bigr|_2
&\leq e^{-\frac{M(n)h\Gamma}{2}}\bigl|\theta^{n} - \theta_{\pi^n}\bigr|_2
\leq R e^{-\frac{M(n)h\Gamma}{2}}.
\end{align}
For any $n \in \mathbb{N}$, this implies that
\begin{align}
\sum_{k=0}^{n-1}\bigl|\theta^{k+1} - \theta_{\pi^k}\bigr|_2
&\leq R \sum_{k=0}^{n-1} \frac{1}{c^2(k+1)^2} = \frac{R}{c^2} \sum_{j=1}^{n} \frac{1}{j^2}.
\end{align}Now using the classical result $\sum_{j=0}^{\infty} \frac{1}{j^2} = \frac{\pi^2}{6} \leq 2$, we obtain 
\begin{equation}
\sum_{k=0}^{n-1}\bigl|\theta^{k+1} - \theta_{\pi^k}\bigr|_2 \leq \frac{2R^2}{c^2}.
\end{equation}
Substituting this into \eqref{eq:main_thm_sub_in}, we conclude the proof.
\end{proof}

\subsection{Proof of Theorem \ref{thm:exponential_conv}}\label{sec:proof_exp_conv}
\begin{proof}
To ease notation let $V^{n} := V^{\pi^n}_{\tau}$ for $n \in \mathbb{N}$ and let $V^* := V^{\pi^*}_{\tau}$. Define $\xi = \frac{1}{1-\gamma}\left|\frac{\mathrm{d} d_{\rho}^{\pi^*}}{\mathrm{d}\rho} \right|_{B_b(S)}$ and note that $\frac{\mathrm{d} d_{\rho}^{\pi}}{\mathrm{d}\rho}(s) \geq (1-\gamma)$ for any $s \in S$ and $\pi \in \Pi_{\mu}$. Hence it holds that 
\begin{equation}
\frac{\mathrm{d} d_{\rho}^{\pi^*}}{\mathrm{d}\rho}(s) = \frac{\mathrm{d} d_{\rho}^{\pi^*}}{\mathrm{d} d_{\rho}^{\pi}}(s)\frac{\mathrm{d} d_{\rho}^{\pi}}{\mathrm{d}\rho}(s) \geq (1-\gamma)\frac{\mathrm{d} d_{\rho}^{\pi^*}}{\mathrm{d} d_{\rho}^{\pi}}(s).
\end{equation}
for any $\pi \in \Pi_{\mu}$ and $s \in S$. Thus we have $\nu^n := \left|\frac{\mathrm{d} d_{\rho}^{\pi^*}}{\mathrm{d} d_{\rho}^{\pi^n}} \right|_{B_b(S)} \leq \xi$ for all $n \in \mathbb{N}$. 
Moreover, recall that by \eqref{eq:mirror_descent_negative} for all $s \in S$ it holds that
\begin{equation}
\int_{A} A(s,a;\theta^{n+1})(\pi^{n+1} - \pi^n)(da|s) + \frac{1}{\lambda}\operatorname{KL}(\pi^{n+1}(\cdot|s)|\pi^{n}(\cdot|s)) \leq 0.
\end{equation}Therefore, it holds that
\begin{align}
&\int_{S} \left(\int_{A} A(s,a;\theta^{n+1})(\pi^{n+1}-\pi^{n})(da|s) + \frac{1}{\lambda}\operatorname{KL}(\pi^{n+1}|\pi^n)(s) \right) d_{\rho}^{\pi^*}(ds) \nonumber\\
&= \int_{S} \left(\int_{A} A(s,a;\theta^{n+1})(\pi^{n+1}-\pi^{n})(da|s) + \frac{1}{\lambda}\operatorname{KL}(\pi^{n+1}|\pi^n)(s) \right) \frac{\mathrm{d} d_{\rho}^{\pi^*}}{\mathrm{d} d_{\rho}^{\pi^{n+1}}}(s) d_{\rho}^{\pi^{n+1}}(ds) \nonumber\\
&\geq \left|\frac{\mathrm{d} d_{\rho}^{\pi^*}}{\mathrm{d} d_{\rho}^{\pi^{n+1}}}\right|_{B_b(S)} 
\int_{S} \left(\int_{A} A(s,a;\theta^{n+1})(\pi^{n+1}-\pi^{n})(da|s) + \frac{1}{\lambda}\operatorname{KL}(\pi^{n+1}|\pi^n)(s) \right)d_{\rho}^{\pi^{n+1}}(ds) \nonumber\\
&\geq \nu^{n+1}\int_{S} \left(\int_{A} A(s,a;\theta^{n+1})(\pi^{n+1}-\pi^{n})(da|s) + \tau\operatorname{KL}(\pi^{n+1}|\pi^n)(s) \right)d_{\rho}^{\pi^{n+1}}(ds) \nonumber\\
&= \nu^{n+1}\int_{S} \left(\int_{A} A^{\pi^n}_{\tau}(s,a)(\pi^{n+1}-\pi^{n})(da|s) + \tau\operatorname{KL}(\pi^{n+1}|\pi^n)(s) \right)d_{\rho}^{\pi^{n+1}}(ds) \nonumber\\
&\qquad + \nu^{n+1}\int_{S} \int_{A} \left(A(s,a;\theta^{n+1}) - A^{\pi^n}_{\tau}(s,a)\right)(\pi^{n+1} - \pi^{n})(da|s) d_{\rho}^{\pi^{n+1}}(ds) \nonumber\\
&=\nu^{n+1}(1-\gamma)\left(V^{n+1}(\rho) - V^{n}(\rho) \right) \nonumber\\
&\qquad+ \nu^{n+1}\int_{S} \int_{A} \left(A(s,a;\theta^{n+1}) - A^{\pi^n}_{\tau}(s,a)\right)(\pi^{n+1} - \pi^{n})(da|s) d_{\rho}^{\pi^{n+1}}(ds),
\end{align}
where the first inequality is due to the integrand in the $s \in S$ variable being non-positive, the second inequality is due to $0 \leq \tau \lambda \leq 1$ and the final equality is then an application of performance difference. Now identically to \eqref{eq:carry_on} from the proof of Theorem \ref{thm:main_up_to_errors}, after an application of the Three Point Lemma, for all $n \in \mathbb{N}$ it holds that
\begin{align}\label{eq:from_before}
&\frac{1}{\lambda}\left(\operatorname{KL}(\pi^*|\pi^{n+1})\left(d_{\rho}^{\pi^*}\right) -  \operatorname{KL}(\pi^*|\pi^n)\left(d_{\rho}^{\pi^*}\right)\right) \nonumber\\
&\leq \int_{S}\int_{A} A(s,a;\theta^{n+1})(\pi^* - \pi^{n})(da|s)d_{\rho}^{\pi^{*}}(ds) \nonumber\\
&\qquad- \left(\int_{S}\int_{A} A(s,a;\theta^{n+1})(\pi^{n+1} - \pi^{n})(da|s)d_{\rho}^{\pi^{*}}(ds) +  \tau \operatorname{KL}(\pi^{n+1}|\pi^n)\left(d_{\rho}^{\pi^{*}}\right)\right).
\end{align}
To upper bound the first term on the right hand side, observe that
\begin{align}
&\int_{S \times A} A(s,a;\theta^{n+1})(\pi^{*} - \pi^{n})(da|s)d_{\rho}^{\pi^*}(ds) \\
&= \int_{S \times A} A^{\pi^n}_{\tau}(s,a)(\pi^{*} - \pi^{n})(da|s)d_{\rho}^{\pi^*}(ds) + \int_{S \times A}\left(A(s,a;\theta^{n+1}) - A^{\pi^n}_{\tau}(s,a) \right) (\pi^{*} - \pi^{n})(da|s)d_{\rho}^{\pi^*}(ds) \nonumber\\
&=(1-\gamma)(V^*(\rho) - V^n(\rho)) - \tau \operatorname{KL}(\pi^*|\pi^n)\left(d_{\rho}^{\pi^*}\right)+ \int_{S \times A}\left(A(s,a;\theta^{n+1}) - A^{\pi^n}_{\tau}(s,a) \right) (\pi^{*} - \pi^{n})(da|s)d_{\rho}^{\pi^*}(ds) 
\end{align}
where we added and subtract the true advantage function and applied the performance difference lemma on the first term in the second equality.
Now observe that 
\begin{align}
&\int_{S \times A} A(s,a;\theta^{n+1})(\pi^{*} - \pi^{n})(da|s)d_{\rho}^{\pi^*}(ds) \\
&= (1-\gamma)(V^*(\rho) - V^n(\rho)) - \tau \operatorname{KL}(\pi^*|\pi^n)\left(d_{\rho}^{\pi^*}\right)+ \int_{S \times A}\left(A(s,a;\theta^{n+1}) - A^{\pi^n}_{\tau}(s,a) \right) (\pi^{*} - \pi^{n})(da|s)\frac{\mathrm{d}d_{\rho}^{\pi^*}}{\mathrm{d}d_{\rho}^{\pi^{n+1}}}(s)d_{\rho}^{\pi^{n+1}}(ds)\\
&\leq (1-\gamma)(V^*(\rho) - V^n(\rho)) - \tau \operatorname{KL}(\pi^*|\pi^n)\left(d_{\rho}^{\pi^*}\right)+2\left|\frac{\mathrm{d} d_{\rho}^{\pi^*}}{\mathrm{d} d_{\rho}^{\pi^{n+1}}}\right|_{B_b(S)}|\theta^{n+1} - \theta_{\pi^n}|_2 \nonumber\\
&=(1-\gamma)(V^*(\rho) - V^n(\rho)) - \tau \operatorname{KL}(\pi^*|\pi^n)\left(d_{\rho}^{\pi^*}\right)+2\nu^{n+1}|\theta^{n+1} - \theta_{\pi^n}|_2,
\end{align}
where we performed a change of measure on the second term in the first equality and used Hölder's inequality along with Assumption \ref{as:bounded_phi} in the final inequality. 
Substituting this into \eqref{eq:from_before}, for all $n \in \mathbb{N}$ it holds that
\begin{align}
&\frac{1}{\lambda}\left(\operatorname{KL}(\pi^*|\pi^{n+1})\left(d_{\rho}^{\pi^*}\right) -  \operatorname{KL}(\pi^*|\pi^n)\left(d_{\rho}^{\pi^*}\right)\right) \nonumber\\
&\leq (1-\gamma)(V^*(\rho) - V^n(\rho))
+ (1-\gamma)\nu^{n+1}(V^{n+1}(\rho) - V^{n}(\rho)) -\tau \operatorname{KL}(\pi^*|\pi^n)\left(d_{\rho}^{\pi^*}\right)\nonumber\\
&\qquad+ 4\nu^{n+1}|\theta^{n+1} - \theta_{\pi^n}|_2.
\end{align}
Now let $y^n = \operatorname{KL}(\pi^*|\pi^n)\left(d_{\rho}^{\pi^*}\right)$ and $\delta^n = V^n(\rho) - V^*(\rho)$. Rearranging, a direct calculations show that for any $n \in \mathbb{N}$
\begin{align}
\delta^{n} - \nu^{n+1}\left(\delta^{n+1} - \delta^n - \frac{4}{1-\gamma}|\theta^{n+1} - \theta_{\pi^n}|_2 \right) &\leq  \frac{1}{\lambda(1-\gamma)}\left((1-\tau\lambda)y^n - y^{n+1} \right).
\end{align}
Now recall that by Lemma \ref{lemma:value_bound}, for all $n \in \mathbb{N}$ it holds that
\begin{equation}
\delta^{n+1} - \delta^{n} = V^{n+1}(\rho) - V^{n}(\rho) \leq \frac{2}{1-\gamma}|\theta^{n+1} - \theta_{\pi^n}|.
\end{equation}Therefore it also holds that $\delta^{n+1} - \delta^n - \frac{4}{1-\gamma}|\theta^{n+1} - \theta_{\pi^n}|_2 \leq 0$. Hence using that $\nu^{n} \leq \xi$ for all $n \in \mathbb{N}$, we have
\begin{align}
&\delta^{n} + \xi\left(\delta^{n+1} - \delta^n - \frac{4}{1-\gamma}|\theta^{n+1} - \theta_{\pi^n}|_2 \right)\leq  \frac{1}{\lambda(1-\gamma)}\left((1-\tau\lambda)y^n - y^{n+1} \right).
\end{align}Thus dividing through by $\xi > 0$ and rearranging again, for all $n \in \mathbb{N}$ it holds that
\begin{align}\label{eq:simplify}
\delta^{n+1} + \frac{1}{\lambda(1-\gamma)\xi}y^{n+1} \leq \frac{\xi - 1}{\xi}\left(\delta^n + \frac{1-\tau\lambda}{\lambda(1-\gamma)(\xi-1)}y^n \right) + \frac{4}{1-\gamma}|\theta^{n+1} - \theta_{\pi^n}|_2.
\end{align}Moreover, by assumption we have that $\xi > 1$ and $\frac{1}{\lambda} < \tau \xi$ which in turn implies that $\frac{1-\tau\lambda}{\lambda(1-\gamma)(\xi-1)} \leq \frac{1}{\lambda(1-\gamma)\xi}$. 
Therefore after simplifying \eqref{eq:simplify}, for all $n \in \mathbb{N}$ it holds that
\begin{align}
\delta^{n+1} + \frac{1}{\lambda(1-\gamma)\xi}y^{n+1} \leq \frac{\xi - 1}{\xi}\left(\delta^n + \frac{1}{\lambda(1-\gamma)\xi}y^n \right) + \frac{4}{1-\gamma}|\theta^{n+1} - \theta_{\pi^n}|_2.
\end{align}Iterating this inequality, we have
\begin{align}
\delta^{n+1} + \frac{1}{\lambda(1-\gamma)\xi}y^{n+1} 
&\leq \left(\frac{\xi - 1}{\xi}\right)^{n+1}\left(\delta^0 + \frac{1}{\lambda(1-\gamma)\xi}y^0 \right) \\
&\quad + \frac{4}{1-\gamma}\sum_{k=0}^{n}\left(\frac{\xi - 1}{\xi}\right)^{n-k}| \theta^{k+1} - \theta_{\pi^{k}}|_2 \\
&\leq \label{eq:exp_conv_errors} e^{-\frac{n+1}{\xi}}\left(\delta^0 + \frac{1}{\lambda(1-\gamma)\xi}y^0 \right) + \frac{4}{1-\gamma}\sum_{k=0}^{n}\left(\frac{\xi - 1}{\xi}\right)^{n-k}| \theta^{k+1} - \theta_{\pi^{k}}|_2,
\end{align} 

where we used the identity $(1-x)^{M} \leq e^{-xM}$ for any $x \in (0,1)$ in the final inequality.
To ease notation, let $\kappa = \frac{\xi -1}{\xi} \in (0,1)$.
Focusing on the error term on the right hand side, by assumption we have that for each policy update $n \in \mathbb{N}$, the number of temporal difference steps satsifies $M:=M(n) \geq \frac{4c}{h\Gamma}(n+1)$ with $c > 0$ the constant from Theorem \ref{thm:bounded_under_log_growth}. 
Therefore by Theorem \ref{thm:bounded_under_log_growth} and Lemma \ref{lemma:everyhing_bdd_if_KL_bdd}, there exists $R > 0$ such that $|\theta^{k} - \theta_{\pi^{k}}|_2 \leq R$ for all $k \in \mathbb{N}$. 
Hence for all $n \in \mathbb{N}$ we have
\begin{align}
\sum_{k=0}^{n}\kappa^{n-k}|\theta^{k+1} - \theta_{\pi^{k}}|_2 &\leq R\sum_{k=0}^{n}\kappa^{n-k}e^{-\frac{M(k)h\Gamma}{2}} \leq \label{eq:reindex}R\sum_{k=0}^{n}\kappa^{n-k}e^{-2c(k+1)}.
\end{align}Simplifying \eqref{eq:reindex}, for all $ n \in \mathbb{N}$ it holds that
\begin{align}
\sum_{k=0}^{n}\kappa^{n-k}e^{-2c(k+1)} = e^{-2c(n+1)}\sum_{k=0}^{n} \kappa^{n-k}e^{2c(n-k)} =e^{-2c(n+1)} \sum_{m=0}^{n}\left(\kappa e^{2c}\right)^{m} = e^{-2c(n+1)}\left(\frac{\left(\kappa e^{2c} \right)^{n+1} - 1}{\kappa e^{2c} - 1}\right),
\end{align}where used the property of the geometric series $\sum_{m=0}^{n} a^m = \frac{a^{n+1} - 1}{a-1}$ for any positive real number $a \neq 1$ and for all $n \in \mathbb{N}$. 
Therefore we have
\begin{align}
\sum_{k=0}^{n}\kappa^{n-k}|\theta^{k+1} - \theta_{\pi^{k}}|_2
&\leq R\sum_{k=0}^{n}\kappa^{n-k}e^{-2c(k+1)} \\
&\leq R e^{-2c(n+1)}\left(\frac{\left(\kappa e^{2c} \right)^{n+1} - 1}{\kappa e^{2c} - 1}\right)\\
&\leq  R e^{-2c(n+1)}\left(\frac{\left(\kappa e^{2c} \right)^{n+1} + 1}{\left|\kappa e^{2c} - 1\right|}\right) \\
&\leq R\frac{\kappa^{n+1} + e^{-2c(n+1)}}
{\left| \kappa e^{2c} - 1\right|} \\
&\leq \frac{R}{{\left| \kappa e^{2c} - 1\right|} }\left(e^{-\frac{n+1}{\xi}} + e^{-2c(n+1)} \right) \\
&\leq \frac{R}{{\left| \kappa e^{2c} - 1\right|} } e^{-\min\left\{\frac{1}{\xi}, 2c \right\}(n+1)},
\end{align}
where in the final inequality we used the standard identity $(1-x)^{M} \leq e^{-xM}$ for any $x \in (0,1)$.
Substituting this into \eqref{eq:exp_conv_errors} we obtain 
\begin{align}
\delta^{n+1} + \frac{1}{\lambda(1-\gamma)\xi}y^{n+1} 
&\leq e^{-\frac{n+1}{\xi}}\left(\delta^0 + \frac{1}{\lambda(1-\gamma)\xi}y^0 \right)  + \frac{4R}{(1-\gamma) |\kappa e^{\frac{c}{\xi}}-1|} e^{-\min\left\{\frac{1}{\xi}, 2c \right\}(n+1)} \\
&\leq e^{-\min\left\{\frac{1}{\xi}, 2c \right\}(n+1)}\left(\delta^0 + \frac{1}{\lambda(1-\gamma)\xi}y^0   + \frac{4R}{(1-\gamma) |\kappa e^{\frac{c}{\xi}}-1|}\right),
\end{align}which concludes the proof.

\end{proof}

\end{document}